\documentclass[final,onefignum,onetabnum]{siamart220329}

\usepackage{amsmath,amsfonts,amsopn,amssymb}
\usepackage{graphicx}
\usepackage[caption=false]{subfig}
\usepackage{multirow,relsize,makecell}
\usepackage{url}
\ifpdf
  \DeclareGraphicsExtensions{.eps,.pdf,.png,.jpg}
\else
  \DeclareGraphicsExtensions{.eps}
\fi
\usepackage{enumitem}
\usepackage{algorithm, algorithmic}

\numberwithin{theorem}{section}
\newsiamremark{remark}{Remark}
\newsiamremark{assumption}{Assumption}
\crefname{remark}{Remark}{Remarks}
\crefname{assumption}{Assumption}{Assumptions}

\headers{Schwarz for fourth-order variational inequalities}{Jongho Park}

\title{Additive Schwarz methods for fourth-order variational inequalities\thanks{Submitted to arXiv.
\funding{This work was supported by Basic Science Research Program through NRF funded by the Ministry of Education~(No.2019R1A6A1A10073887).}
}}

\author{Jongho Park\thanks{Computer, Electrical and Mathematical Sciences and Engineering Division, King Abdullah University of Science and Technology~(KAUST), Thuwal 23955, Saudi Arabia 
 (\email{jongho.park@kaust.edu.sa}, \url{https://sites.google.com/view/jonghopark}).}}

\ifpdf
\hypersetup{
  pdftitle={Additive Schwarz methods for fourth-order variational inequalities},
  pdfauthor={Jongho Park}
}
\fi

\newcommand\gap{\hspace{0.1cm}}
\newcommand\smallgap{\hspace{0.05cm}}

\newcommand\cI{\mathcal{I}}
\newcommand\cT{\mathcal{T}}
\newcommand\cV{\mathcal{V}}
\newcommand\rT{\mathrm{T}}
\newcommand\tS{\widetilde{S}}

\newcommand\un{u^{(n)}}
\newcommand\unn{u^{(n+1)}}

\newcommand\intO{\int_{\Omega}}
\newcommand\sumk{\sum_{k=1}^N}
\newcommand\sumT{\sum_{T \in \mathcal{T}_h}}
\newcommand\Rw{R_k^* w_k}

\DeclareMathOperator*{\argmin}{\arg\min}
\DeclareMathOperator{\diam}{diam}

\begin{document}

\maketitle

\begin{abstract}
Fourth-order variational inequalities are encountered in various scientific and engineering disciplines, including elliptic optimal control problems and plate obstacle problems.
In this paper, we consider additive Schwarz methods for solving fourth-order variational inequalities.
Based on a unified framework of various finite element methods for fourth-order variational inequalities, we develop one- and two-level additive Schwarz methods.
We prove that the two-level method is scalable in the sense that the convergence rate of the method depends on $H/h$ and $H/\delta$ only, where $h$ and $H$ are the typical diameters of an element and a subdomain, respectively, and $\delta$ measures the overlap among the subdomains.
This proof relies on a new nonlinear positivity-preserving coarse interpolation operator, the construction of which was previously unknown.
To the best of our knowledge, this analysis represents the first investigation into the scalability of the two-level additive Schwarz method for fourth-order variational inequalities.
Our theoretical results are verified by numerical experiments.
\end{abstract}



\section{Introduction}
\label{Sec:Introduction}
This paper is concerned with scalable parallel algorithms to solve fourth-order variational problems.
Let $\Omega \subset \mathbb{R}^2$ be a convex polygonal domain and $S$ be a subspace of $H^2 (\Omega)$ with a boundary condition that makes the Poincar\'{e}--Friedrichs inequality hold~\cite{Necas:2012}.
As a model problem, we consider the constrained optimization problem
\begin{equation}
\label{model_opt}
\min_{v \in K} \left\{ \frac{1}{2} a(v,v) - (f,v) \right\},
\end{equation}
where $a ( \cdot, \cdot)$ is a continuous and coercive bilinear form on $S$ derived from a fourth-order elliptic problem, $(\cdot, \cdot)$ is the standard $L^2 (\Omega)$-inner product, and the constraint set $K$ is given by
\begin{equation*}
K = \left\{ v \in S : v \leq \psi \textrm{ in } \Omega \right\}
\end{equation*}
for some $\psi \in C (\overline{\Omega}) \cap C^2 (\Omega)$.
It is well-known that~\eqref{model_opt} admits an equivalent fourth-order variational inequality formulation~\cite{BS:2017,BSZ:2012}: find $u \in K$ such that
\begin{equation}
\label{model_VI}
a(u, v - u) \geq (f, v - u) \quad \forall v \in K.
\end{equation}
Since~\eqref{model_opt} and~\eqref{model_VI} are equivalent, we refer to the optimization problem~\eqref{model_opt} as a fourth-order variational inequality as well throughout this paper.

Problems of the form either~\eqref{model_opt} or~\eqref{model_VI} appear in diverse fields of science and engineering.
For example, if we set
\begin{equation}
\label{PL}
S = H_0^2 (\Omega), \quad
a(v, w) = \intO \nabla^2 v : \nabla^2 w \,dx,
\end{equation}
where $\nabla^2$ and $:$ denote the Hessian and the Frobenius inner product, respectively, then we obtain the displacement obstacle problem of clamped plates~\cite{BDS:2018,BSW:2022,BSZ:2012}.
On the other hand, if we set
\begin{equation}
\label{OC}
S = H^2 (\Omega) \cap H_0^1 (\Omega), \quad
a(v, w) = \intO (\beta \nabla^2 v : \nabla^2 w + v w) \,dx 
\end{equation}
for some $\beta > 0$, then we get the elliptic distributed optimal control problem with pointwise constraints introduced in~\cite{BS:2017,Casas:1986,LGY:2009}.
Various finite element discretizations for either~\eqref{model_opt} or~\eqref{model_VI} have been considered in a number of existing works: nonconforming methods~\cite{LGY:2009}, mixed methods~\cite{GY:2011}, partition of unity methods~\cite{BDS:2014}, discontinuous Galerkin methods~\cite{BGS:2018,CZ:2019}, and virtual element methods~\cite{BST:2021}.
In particular, unified frameworks for finite element methods including $C^1$ conforming methods, nonconforming methods, and interior penalty methods were proposed in~\cite{BS:2017,BSZ:2012}.

We consider domain decomposition methods~\cite{TW:2005} as parallel numerical solvers for the fourth-order variational inequality~\eqref{model_opt}.
Our motivation stems from the literature on domain decomposition methods for optimization problems of the form~\eqref{model_opt} and variational inequalities of the form~\eqref{model_VI}.
Schwarz methods for second-order variational inequalities of the form~\eqref{model_VI} were studied in~\cite{BTW:2003,BW:2000,Carstensen:1997,Tai:2003}, and subsequently generalized to variational inequalities of the second kind and quasi-variational inequalities in~\cite{BK:2012}.
A dual-primal nonoverlapping domain decomposition method for variational inequalities appearing in structural mechanics was proposed in~\cite{LP:2021}.
Meanwhile, convergence theory of Schwarz methods for smooth convex optimization problems was developed in~\cite{TX:2002}, and then extended to constrained and nonsmooth problems in~\cite{Badea:2006} and~\cite{Park:2020}, respectively.

In this paper, we present and analyze additive Schwarz methods for solving the fourth-order variational inequality~\eqref{model_opt}.
In contrast to several existing works~\cite{BDS:2018,BSW:2022} focusing on auxiliary linear problems within the primal-dual active set method~\cite{BIK:1999,HIK:2002} for~\eqref{model_opt} and~\eqref{model_VI}, the methods considered here are based on the nonlinear subspace correction framework for convex optimization problems as presented in~\cite{Badea:2006,LP:2022,Park:2020,TX:2002}.
That is, each subdomain problem in the proposed methods is nonlinear and has the same form as the full-dimensional problem~\eqref{model_opt}.
This approach leads to globally linear convergent algorithms, contrasting with the fast but locally convergent nature of the primal-dual active set method.

Based on a unified framework~\cite{BS:2017} of various finite element methods for~\eqref{model_opt} including $C^1$ finite element methods~\cite{Zhang:1996}, nonconforming finite element methods~\cite{Brenner:1996}, and interior penalty methods~\cite{BW:2005}, we investigate a one-level additive Schwarz method and prove that its additive Schwarz condition number~\cite{Park:2020} is $O(1/ H \delta^3)$, where $H$ denotes the subdomain diameter and $\delta$ measures the overlap among the subdomains.
This estimate is the same as those of one-level additive Schwarz methods for the auxiliary linear problems considered in~\cite{BDS:2018,BSW:2022}.
In addition, inspired by the partition of unity method for fourth-order elliptic problems~\cite{BDS:2014,BDS:2018}, we introduce a novel coarse space suitable for the constrained problem~\eqref{model_opt}.
We show that the additive Schwarz condition number of a two-level additive Schwarz method equipped with the proposed coarse space is bounded by $O( C(H,h) (H/\delta)^3 )$, where $h$ is the diameter of fine elements and $C (H,h)$ a positive constant depending on $H/h$ only.
Consequently, the two-level method exhibit scalability in the sense that its global linear convergence rate depends on $H/h$ and $H/\delta$ only.
In the convergence analysis of the two-level method, we rely on a novel nonlinear positivity-preserving coarse interpolation operator.
Although the construction of such an operator for second-order variational inequalities has been successfully addressed in various existing works~\cite{Badea:2006,BTW:2003,Tai:2003}, extending this construction to fourth-order problems like~\eqref{model_opt} has been regarded as a challenging task~(see, e.g.,~\cite{NW:2001,Tai:2005}). 
Our construction provides, to the best of our knowledge, the first solution to this previously unsolved problem.

This paper is organized as follows.
In Section~\ref{Sec:Pre}, we introduce finite element discretizations and additive Schwarz methods for the fourth-order variational inequality~\eqref{model_opt}.
In Section~\ref{Sec:1L}, we present a one-level additive Schwarz method for~\eqref{model_opt}.
In Section~\ref{Sec:2L}, we propose a novel partition of unity coarse space for~\eqref{model_opt} and analyze a two-level additive Schwarz method equipped with the proposed coarse space.
In Section~\ref{Sec:Solvers}, we discuss local and coarse problems in the additive Schwarz methods. 
In Section~\ref{Sec:Numerical}. we present numerical results that support our theoretical results.
In Section~\ref{Sec:J_i}, we present a proof of a key lemma~(see Lemma~\ref{Lem:J_i}) for the construction of the nonlinear positivity-preserving coarse interpolation operator.
We conclude this paper with remarks in Section~\ref{Sec:Conclusion}.

\section{Preliminaries}
\label{Sec:Pre}
In this section, we provide preliminaries for this paper.
We introduce finite element discretizations of the fourth-order variational inequality~\eqref{model_opt}.
Then, we present a general additive Schwarz method for the fourth-order variational inequality based on an abstract space decomposition.

In what follows, we use the notation $A \lesssim B$ and $B \gtrsim A$ to represent that there exists a constant $c > 0$ such that $A \leq c B$, where $c$ is independent of the parameters $h$, $H$, and $\delta$ relying on discretization and domain decomposition.
The notation $A \approx B$ means that $A \lesssim B$ and $A \gtrsim B$.

\subsection{Finite element discretizations}
We present a unified framework of finite element methods for the fourth-order variational inequality~\eqref{model_opt}, where similar frameworks were introduced in~\cite{BS:2017,BSZ:2012}.
This framework can deal with various finite element methods for fourth-order problems such as $C^1$ finite element methods, nonconforming finite element methods, and interior penalty methods.
One may refer to~\cite{Zhang:1996} and~\cite{Brenner:1996} for various examples of $C^1$ conforming and nonconforming finite elements, respectively.

Let $\cT_h$ be a quasi-uniform triangulation\footnotemark[1]\footnotetext[1]{In some finite elements for~\eqref{model_opt} such as the Bogner--Fox--Schmit element~\cite{BFS:1965,Valdman:2020}, the reference element is not a triangle but a quadrilateral.
Nevertheless, for the sake of convenience, we use the terminology ``triangulation" in these cases as well.} of $\Omega$ with $h$ the characteristic element diameter.
Assume that we have the following:
\begin{itemize}
\item a finite element space $S_h \subset L^2 (\Omega)$ defined on $\cT_h$,
\item a norm $\| \cdot \|_h$ defined on $S_h + S$,
\item a symmetric bilinear form $a_h (\cdot, \cdot)$ on $S_h$,
\item an enriching operator $E_h$ that maps $S_h$ to a conforming finite element space $\tS_h \subset S$.
\end{itemize}
Note that $S_h$ is not necessarily a subspace of $S$, but the norm $| \cdot |_h$ is well-defined for functions in both $S_h$ and $S$.
We also assume that every function in $S_h$ is continuous at the vertices of $\cT_h$ so that the following constraint set $K_h \subset S_h$ is well-defined:
\begin{equation*}
K_h = \left\{ v \in S_h : v \leq \psi \textrm{ at the vertices of } \cT_h \right\}.
\end{equation*}
Under the above assumptions, the following finite element approximation of~\eqref{model_opt} defined on $S_h$ can be considered:
\begin{equation}
\label{model_FEM}
\min_{v \in K_h} \left\{ F_h (v) := \frac{1}{2} a_h (v, v)  -(f, v) \right\}.
\end{equation}

In the convergence analysis of the additive Schwarz methods to be introduced in this paper, we rely on several assumptions regarding $\| \cdot \|_h$, $a_h (\cdot, \cdot)$, and $E_h$.
These assumptions are summarized below.

\begin{assumption}
\label{Ass:FEM}
In the finite element discretization~\eqref{model_FEM}, we have the following:
\begin{enumerate}[label=(\alph*)]
    \item The norm $\| \cdot \|_h$ and the bilinear form $a_h (\cdot, \cdot)$ are equivalent in $S_h$, i.e.,
    \begin{equation}
    \label{norm}
        a_h (v, v) \approx \| v \|_h^2, \quad v \in S_h.
    \end{equation}
    \item The enriching operator $E_h \colon S_h \rightarrow \tS_h$ satisfies
    \begin{equation}
    \label{E_h}
        \| v - E_h v \|_{L^2 (\Omega)} + h \left( \sumT | v - E_h v |_{H^1 (T)}^2 \right)^{\frac{1}{2}} + h^2 | E_h v |_{H^2 (\Omega)}
        \lesssim h^2 \| v \|_h, \text{ } v \in S_h.
    \end{equation}
    Moreover, it preserves the function values at the vertices of $\cT_h$~:
    \begin{equation}
    \label{E_h_preserving}
        E_h v = v \text{ at the vertices of } \cT_h, \quad v \in S_h.
    \end{equation}
\end{enumerate}
\end{assumption}

The assumptions presented in Assumption~\ref{Ass:FEM} are valid for various finite element methods for~\eqref{model_opt} encompassing those discussed in~\cite{BS:2017,BSZ:2012}.
An instance of~\eqref{norm} can be found in~\cite[equations~(3.4) and~(3.5)]{BS:2017}.
Moreover, conditions~\eqref{E_h} and~\eqref{E_h_preserving} coincide with those presented in~\cite[equations~(3.10) and~(3.11)]{BS:2017}.
As discussed in~\cite[Remark~3.2]{BS:2017} and~\cite[Examples~2.2--2.4]{BSZ:2012}, for $C^1$ finite element methods, $E_h$ is the natural injection operator, whereas for nonconforming finite element methods, $E_h$ is defined by averaging~\cite{Brenner:1996}.

Thanks to~\eqref{norm}, the energy functional $F_h$ of~\eqref{model_FEM} is strongly convex with respect to the $\| \cdot \|_h$-norm.
Hence,~\eqref{model_FEM} admits a unique solution $u_h \in K_h$.

\begin{remark}
\label{Rem:FEM_error}
To derive an error estimate between $u_h$ and the solution of the continuous problem~\eqref{model_opt}, additional assumptions such as elliptic regularity, interpolation estimates, and the approximability of $a_h (\cdot, \cdot)$ for $a (\cdot, \cdot)$ are typically needed, as discussed in~\cite{BS:2017,BSZ:2012}.
However, since these assumptions are not necessary for the convergence analysis of additive Schwarz methods, we do not include them here.
\end{remark}

\subsection{Additive Schwarz method}
We introduce a general additive Schwarz method for the discrete fourth-order variational inequality~\eqref{model_FEM} based on an abstract space decomposition for the solution space $V = S_h$.
We assume that there exist finite-dimensional spaces $V_k$, $1 \leq k \leq N$, and injective linear operators $R_k^* \colon V_k \rightarrow V$ such that
\begin{equation}
\label{1L}
V = \sumk R_k^* V_k.
\end{equation}
Note that $V_k$ need not be a subspace of $V$.
A stable decomposition assumption~\cite[Assumption~4.1]{Park:2020} associated with~\eqref{1L} for the problem~\eqref{model_FEM} is stated below.

\begin{assumption}[stable decomposition]
\label{Ass:stable}
There exists a constant $C_0$ such that, for any $v, w \in V$ with $v, v+w \in K_h$, there exists a decomposition
\begin{equation*}
w = \sumk \Rw, \quad w_k \in V_k, \gap 1 \leq k \leq N,
\end{equation*}
that satisfies $v + \Rw \in K_h$ and
\begin{equation*}
\sumk a_h (\Rw, \Rw ) \leq \frac{C_0^2}{2} \| w \|_h^2. 
\end{equation*}
\end{assumption} 

The abstract additive Schwarz method for~\eqref{model_FEM} under the space decomposition~\eqref{1L} is presented in Algorithm~\ref{Alg:ASM}.
We note that the same algorithm appeared in several existing works~\cite{BW:2000,Park:2020,THX:2002}.
The constant $N_c$ in Algorithm~\ref{Alg:ASM} is the minimum number of classes that is required to classify the spaces $\{ V_k \}_{k=1}^N$ by a usual coloring technique~\cite[Section~5.1]{Park:2020}.

\begin{algorithm}
\caption{Abstract additive Schwarz method for~\eqref{model_FEM}}
\label{Alg:ASM}
\begin{algorithmic}[]
\STATE Choose $u^{(0)} \in K_h$ and $\tau \in (0, 1/N_c ]$.
\FOR{$n=0,1,2,\dots$}
\item \begin{equation*}
\begin{split}
w_k^{(n+1)} &= \argmin_{\substack{w_k \in V_k,\\ \un + \Rw \in K_h}} F_h \left( \un + \Rw \right), \quad 1 \leq k \leq N, \\
\unn &= \un + \tau \sumk \Rw^{(n+1)}
\end{split}
\end{equation*}
\ENDFOR
\end{algorithmic}
\end{algorithm}

Invoking~\eqref{norm} and Assumption~\ref{Ass:stable}, one can prove the linear convergence of Algorithm~\ref{Alg:ASM}; we state the convergence theorem by Park~\cite[Theorem~4.8]{Park:2020} in a form suitable for our purposes.

\begin{theorem}
\label{Thm:conv}
Assume that Assumptions~\ref{Ass:FEM} and~\ref{Ass:stable} holds.
In Algorithm~\ref{Alg:ASM}, we have $\{ u^{(n)} \} \subset K_h$ and
\begin{equation}
\label{Thm1:conv}
F_h (u^{(n)}) - F_h (u_h)
\leq \left( 1 - \frac{\tau}{2} \min \left\{ 1, \frac{1}{2C_0^2} \right\} \right)^n \left( F_h (u^{(0)}) - F_h (u_h) \right),
\text{ } n \geq 0.
\end{equation}
\end{theorem}
\begin{proof}
We first observe that the problem~\eqref{model_FEM} is equivalent to the following composite optimization problem:
\begin{equation}
\label{composite}
\min_{v \in S_h} \left\{ \widetilde{F}_h (v) := F_h (v) + \chi_{K_h} (v) \right\},
\end{equation}
where $\chi_{K_h} \colon S_h \rightarrow \overline{\mathbb{R}}$ is the indicator function of $K_h$, which is given by
\begin{equation*}
    \chi_{K_h} (v) = \begin{cases}
    0, &\quad \text{ if } v \in K_h, \\
    \infty, &\quad \text{ if } v \not\in K_h.
    \end{cases}
\end{equation*}
Note that, for $v \in S_h$, $\widetilde{F}_h (v) <\infty$ if and only if $v \in K_h$.

In~\cite{Park:2020}, four conditions were considered to ensure the linear convergence of the general additive Schwarz method for composite optimization problems of the form~\eqref{composite}: stable decomposition, strengthened convexity, local stability, and sharpness~\cite[Assumptions~4.1--4.3 and~3.4]{Park:2020}.
Among them, strengthened convexity and local stability are evident.
Sharpness of the functional $\widetilde{F}_h$ follows from~\eqref{norm}.
Therefore, verifying Assumption~\ref{Ass:stable}, the stable decomposition condition, is sufficient to satisfy the requirements of~\cite[Theorem~4.8]{Park:2020}.
In conclusion, Assumptions~\ref{Ass:FEM} and~\ref{Ass:stable}, invoking~\cite[Theorem~4.8]{Park:2020} yields
\begin{equation*}
\widetilde{F}_h (u^{(n)}) - \widetilde{F}_h (u_h)
\leq \left( 1 - \frac{\tau}{2} \min \left\{ 1, \frac{1}{2C_0^2} \right\} \right)^n \left( \widetilde{F}_h (u^{(0)}) - \widetilde{F}_h (u_h) \right),
\text{ } n \geq 0,
\end{equation*}
which implies that $\{ u^{(n)} \} \subset K_h$ and establishes~\eqref{Thm1:conv}.
\end{proof}

Thanks to Theorem~\ref{Thm:conv}, it suffices to prove Assumption~\ref{Ass:stable} in order for the convergence analysis of Algorithm~\ref{Alg:ASM}.
In Sections~\ref{Sec:1L} and~\ref{Sec:2L}, we will estimate the constant $C_0^2$ in Assumption~\ref{Ass:stable} corresponding to one- and two-level domain decomposition settings in terms of the parameters $h$, $H$, and $\delta$, respectively.

\begin{remark}
\label{Rem:Acc}
The practical performance of Algorithm~\ref{Alg:ASM} can be enhanced by incorporating acceleration schemes designed for first-order methods in convex optimization.
Combining Algorithm~\ref{Alg:ASM} with an acceleration scheme like FISTA~\cite{BT:2009} with adaptive restart~\cite{OC:2015} results in the accelerated additive Schwarz method proposed in~\cite{Park:2021}.
Moreover, the concept of full backtracking introduced in~\cite{CC:2019} can be integrated with Algorithm~\ref{Alg:ASM} to improve the convergence rate~\cite{Park:2022}.
However, these accelerated variants are not covered here as they are beyond the scope of this paper.
\end{remark}

\section{One-level method}
\label{Sec:1L}
In this section, we present a one-level additive Schwarz method for the discrete fourth-order variational inequality~\eqref{model_FEM} based on an overlapping domain decomposition setting.
Then we provide a convergence analysis of the one-level method.
The convergence result given in this section is applicable to any finite element methods that can be interpreted in the unified framework introduced in Section~\ref{Sec:Pre}.

\subsection{Domain decomposition}
\label{Subsec:DD}
Let $\cT_H$ be a quasi-uniform triangulation of $\Omega$ such that $\cT_h$ is a refinement of $\cT_H$, where $H$ stands for the characteristic element diameter of $\cT_H$.
Two triangulations $\cT_h$ and $\cT_H$ will play roles of fine and coarse meshes, respectively.
We assume that the domain $\Omega$ is decomposed into a collection of overlapping subdomains $\{ \Omega_k \}_{k=1}^N$ such that each $\Omega_k$ is a union of elements in $\cT_h$ and $\diam \Omega_k \approx H$.
The overlap width among the subdomains is denoted by $\delta$.


In Algorithm~\ref{Alg:ASM}, we set $V = S_h$ and
\begin{equation}
\label{V_k}
V_k = S_h (\Omega_k), \quad 1 \leq k \leq N,
\end{equation}
where $S_h (\Omega_k)$ is the finite element space of the same type as $S_h$ defined on $\Omega_k$.
The operator $R_k^* \colon V_k \rightarrow V$, $1 \leq k \leq N$, is given by the natural extension operator from $S_h (\Omega_k)$ to $S_h$.
Then we obtain the desired one-level additive Schwarz method for~\eqref{model_FEM}.
We summarize below some essential assumptions on the local spaces $\{ V_k \}_{k=1}^N$ required for the convergence analysis.

\begin{assumption}
\label{Ass:local}
In the local spaces~\eqref{V_k}, we have the following:
\begin{enumerate}[label=(\alph*)]
    \item The local spaces $\{ V_k \}_{k=1}^N$ can be colored with a number of colors $N_c$ independent of $N$.
    \item For any $w \in V$, there exists a decomposition $w = \sum_{k=1}^N R_k^* w_k$, $w_k \in V_k$, such that
    \begin{equation*}
    \begin{cases}
    \displaystyle 0 \leq w_k \leq w, & \quad \text{ if } w > 0, \\
    \displaystyle w \leq w_k \leq 0, & \quad \text{ if } w < 0, \\
    \displaystyle w_k = 0, & \quad \text{ if } w = 0, \\
    \end{cases}
    \quad \text{at the } \cT_h\text{-vertices in } \overline{\Omega}_k,
    \end{equation*}
    and
    \begin{multline}
    \label{Ass1:local}
        \sumk \| R_k^* w_k \|_h^2 \lesssim \| w \|_h^2 + \frac{1}{\delta^2} \sumT | w |_{H^1 (T)}^2 \\
        + \frac{1}{H\delta^3} \| E_h w \|_{L^2 (\Omega)}^2 + \left( \frac{H}{\delta} \right)^3 | E_h w |_{H^2 (\Omega)}^2.
    \end{multline}
\end{enumerate}
\end{assumption}

As discussed in Section~\ref{Sec:Pre}, Assumption~\ref{Ass:local}(a) can be verified using a standard coloring argument.
Regarding Assumption~\ref{Ass:local}(b), it is typically derived by utilizing a $W^{2,\infty}$-partition of unity~\cite{Brenner:1996,BW:2005,Zhang:1996} subordinate to the domain decomposition $\{ \Omega_k \}_{k=1}^N$ and a specific nodal interpolation operator.
In particular, the last two terms in~\eqref{Ass1:local} can be obtained by a trace theorem-type argument introduced in~\cite{DW:1994}.
One may refer to~\cite[Section~8]{Brenner:1996} and~\cite[Section~5]{BW:2005} for details for nonconforming finite element methods and interior penalty methods, respectively; see also~\cite[Remark~5.3]{BDS:2018} and~\cite[Remark~4.2]{BSW:2022}.

\subsection{Convergence analysis}
As we discussed in Section~\ref{Sec:Pre}, it suffices to estimate the constant $C_0$ in Assumption~\ref{Ass:stable} in order to analyze the convergence rate of the one-level additive Schwarz method.
By a similar argument as~\cite[Theorem~5.1]{BDS:2018} and~\cite[Lemma~4.3]{BSW:2022}, one can prove Theorem~\ref{Thm:1L}, which provides an estimate for $C_0$.

\begin{theorem}
\label{Thm:1L}
Suppose that the following conditions hold:
\begin{itemize}
\item The finite element discretization~\eqref{model_FEM} satisfies Assumption~\ref{Ass:FEM}.
\item The local spaces~\eqref{V_k} satisfy Assumption~\ref{Ass:local}.
\end{itemize}
Then, in the one-level domain decomposition~\eqref{1L}, Assumption~\ref{Ass:stable} is satisfied with
\begin{equation*}
C_0^2 \approx \frac{1}{H \delta^3}.
\end{equation*}
\end{theorem}
\begin{proof}
Take any $v, w \in V$ with $v, v+w \in K_h$.
By Assumption~\ref{Ass:local}, we have $w_k \in V_k$, $1 \leq k \leq N$, such that $w = \sum_{k=1}^N R_k^* w_k$, $v + R_k^* w_k \in K_h$, and
\begin{multline}
\label{Thm1:1L}
    \sumk a_h (\Rw, \Rw) \approx \sumk \| \Rw \|_h^2 \\
    \lesssim \| w \|_h^2 + \frac{1}{\delta^2} \sumT | w |_{H^1 (T)}^2 + \frac{1}{H \delta^3} \| E_h w \|_{L^2 (\Omega)}^2 + \left( \frac{H}{\delta} \right)^3 | E_h w |_{H^2 (\Omega)}^2,
\end{multline}
where the $\approx$-relation is due to~\eqref{norm}.
By the triangle inequality and~\eqref{E_h}, we obtain
\begin{equation}
\label{Thm2:1L}
    \sumT | w |_{H^1 (T)}^2 \lesssim \sumT | w - E_h w |_{H^1 (T)}^2 + | E_h w |_{H^1 (\Omega)}^2
    \lesssim h^2 \| w \|_h^2 + | E_h w |_{H^1 (\Omega)}^2.
\end{equation}
Invoking the Poincar\'{e}--Friedrichs inequality for $S$~(see, e.g.,~\cite[Theorem~1.10]{Necas:2012} for the case $S = H^2 (\Omega) \cap H_0^1 (\Omega)$) and~\eqref{E_h} yields
\begin{equation}
\label{Thm3:1L}
    \| E_h w \|_{L^2 (\Omega)}^2 + | E_h w |_{H^1 (\Omega)}^2 \lesssim | E_h w |_{H^2 (\Omega)}^2 \lesssim \| w \|_h^2.
\end{equation}
Combining~\eqref{Thm1:1L},~\eqref{Thm2:1L} and~\eqref{Thm3:1L} yields
\begin{equation*}
    \sum_{k=1}^N a_h (\Rw, \Rw)
    \lesssim \left[ 1 + \frac{1}{\delta^2} + \frac{1}{H \delta^3} + \left( \frac{H}{\delta} \right)^3 \right] \| w \|_h^2 \lesssim \frac{1}{H\delta^3} \| w \|_h^2,
\end{equation*}
which completes the proof.
\end{proof}

Theorem~\ref{Thm:1L} implies that the one-level additive Schwarz method is not scalable in the sense that $1/H\delta^3$ increases as the number of subdomains increases.
That is, the larger the number of subdomains, the more iterations are needed in the one-level method.
In order to achieve the scalability, we will deal with how to design an appropriate coarse-level correction in Section~\ref{Sec:2L}.

\section{Two-level method}
\label{Sec:2L}
In this section, we develop a two-level additive Schwarz method for~\eqref{model_FEM} by introducing a novel coarse space that is suitable for fourth-order variational inequalities.
We notice that designing appropriate coarse spaces for higher-order variational inequalities has remained as an open problem for a couple of decades.
In particular, it was proven in~\cite{NW:2001} that linear positivity-preserving interpolation operators for higher-order finite elements, with sufficient accuracy, do not exist; preserving positivity is crucial for designing multilevel methods for constrained problems~\cite{Tai:2003,Tai:2005}.
While a nonlinear interpolation operator that locally preserves linear functions was proposed in~\cite[Section~3]{Tai:2005}, it does not ensure positivity-preserving.
Here, we propose a novel positivity-preserving interpolation operator and utilize it in designing a scalable two-level method.

In the two-level method, we use a space decomposition
\begin{equation}
\label{2L}
V = R_0^* V_0 + \sumk R_k^* V_k,
\end{equation}
where $V_0$ is a finite-dimensional space that plays a role of the coarse space, and $R_0^* \colon V_0 \rightarrow V$ is an injective linear operator.
The two-level additive Schwarz method based on~\eqref{2L} is the same as Algorithm~\ref{Alg:ASM} except that the index $k$ runs from $0$ to $N$, so that we do not present it separately.

\subsection{Coarse space}
Let $\{ x^i \}_{i \in \cI_H}$ denote the collection of all vertices of $\cT_H$.
For each $i \in \cI_H$, we define a region $\omega_i \subset \Omega$ as the union of the coarse elements having $x^i$ as a vertex:
\begin{equation}
    \label{omega_i}
    \overline{\omega}_i = \bigcup_{T \in \cT_H, \smallgap x^i \in \partial T} \overline{T}.
\end{equation}
Clearly, $\{ \omega_i \}_{i \in \cI_H}$ forms an overlapping domain decomposition for $\Omega$.
One can explicitly find a $W^{2,\infty}$-partition of unity $\{ \phi_i \}_{i \in \cI_H}$ subordinate to $\{ \omega_i \}_{i \in \cI_H}$ such that
\begin{subequations}
\label{pou_basis}
\begin{align}
\phi_i \geq 0 \quad \text{ on } \omega_i,& \label{pou_basis0} \\
\phi_i = 0 \quad \text{ on } \Omega \setminus \omega_i,& \label{pou_basis1} \\
\sum_{i \in \cI_H} \phi_i = 1 \quad \text{ on } \overline{\Omega}, \label{pou_basis2}& \\
\| \nabla \phi_i \|_{L^{\infty} (\omega_i)} \lesssim \frac{1}{ H}, \gap
\| \nabla^2 \phi_i \|_{L^{\infty} (\omega_i)} \lesssim \frac{1}{H^2},& \quad i \in \cI_H, \label{pou_basis3}
\end{align}
\end{subequations}
For example, if $\cT_H$ consists of triangular elements, then we can set $\phi_i$ as either the fifth-degree Argyris~\cite{AFS:1968} or the Bell~\cite{Bell:1969} basis function at $x^i$ corresponding to function evaluation.
On the other hand, if each $\omega_i$ is rectangular, we can set $\phi_i$ as the Bogner--Fox--Schmit~\cite{BFS:1965,Valdman:2020} basis function at $x^i$ corresponding to function evaluation.
The coarse space $V_0$ is defined as follows:
\begin{equation}
    \label{V_0}
    V_0 = \left\{ v \in \sum_{i \in \cI_H} \phi_i \mathbb{P}_1 (\omega_i) : v \text{ satisfies the boundary condition of }S \right\},
\end{equation}
where $\mathbb{P}_1 (\omega_i)$ is the collection of linear polynomials defined on $\omega_i$.
Clearly, we have $V_0 \subset S$.
Note that $V_0$ can be regarded as a finite element space on $\cT_H$ defined in terms of the partition of unity method~\cite{BDS:2014}.

Since $V_0 \not\subset V$ in general, it is necessary to define an appropriate coarse prolongation operator $R_0^* \colon V_0 \rightarrow V$ that possesses certain approximation properties.
More precisely, the following assumption is required:

\begin{assumption}
\label{Ass:coarse}
The coarse prolongation operator $R_0^* \colon V_0 \rightarrow V$ satisfies
\begin{equation}
\label{intergrid}
\| v - R_0^* v \|_{L^2 (\Omega)} + H \left( \sumT | v - R_0^* v |_{H^1 (T)}^2 \right)^{\frac{1}{2}} + H^2 \| R_0^* v \|_h \\
\lesssim H^2 | v |_{H^2 (\Omega)},
\quad v \in V_0.
\end{equation}
\end{assumption}

One may refer to~\cite[Lemma~3.2]{Zhang:1996},~\cite[Lemma~6.1]{Brenner:1996}, and~\cite[Lemma~3.3]{BW:2005} for coarse prolongation operators satisfying~\eqref{intergrid} for coarse spaces defined in terms of $C^1$ conforming finite element methods, nonconforming finite element methods, and interior penalty methods, respectively.
With Assumption~\ref{Ass:coarse}, we complete the characterization of the proposed two-level method.

\subsection{Convergence analysis}
As we discussed above, the convergence analysis of the two-level method requires to find a stable decomposition stated in Assumption~\ref{Ass:stable} for~\eqref{2L}.
Due to the presence of constraints, finding stable decompositions for two-level methods for variational inequalities is more challenging than that for linear problems.
In particular, as in the existing works~\cite{Badea:2006,BTW:2003,Tai:2003} for second-order problems, we need to find an appropriate coarse interpolation operator onto $V_0$ that preserves positivity, i.e., the interpolation of a nonnegative function is nonnegative.

Throughout this section, we require several additional assumptions on $\cT_h $ and $\cT_H$.

\begin{assumption}
\label{Ass:cT}
The two triangulations $\cT_h$ and $\cT_H$ satisfy the following:
\begin{enumerate}[label=(\alph*)]
\item Each $\omega_i$, $i \in \cI_H$, is convex.
\item There exists an acute angle $\alpha (H,h) \in (0, \pi/2)$ that depends on $H/h$ only such that, for any $i \in \cI_H$, the sine of the angle formed by any three noncollinear $\cT_h$-vertices in $\omega_i$ is bounded below by $\sin \alpha (H,h)$, where $\omega_i$ was defined in~\eqref{omega_i}.
\end{enumerate}
\end{assumption}

Note that Assumption~\ref{Ass:cT}(b) concerns the minimum angle formed by noncollinear vertices of $\cT_h$. Proposition~\ref{Prop:angle} presents an example that satisfies Assumption~\ref{Ass:cT}.

\begin{proposition}
\label{Prop:angle}
If $\cT_h$ and $\cT_H$ are uniform rectangular meshes, then Assumption~\ref{Ass:cT} holds with
\begin{equation*}
\sin \alpha (H,h) \approx \left( \frac{H}{h} \right)^{-2}.
\end{equation*}
\end{proposition}
\begin{proof}
As Assumption~\ref{Ass:cT}(a) is clear, we focus on Assumption~\ref{Ass:cT}(b).
We may assume that each element in $\cT_h$ is a square with the side length $h$ and that each $\omega_i$ is a square with the side length $H$.
Then $\omega_i$ is composed of $m \times m$ $\cT_h$-elements, where $m = H/h$.
Since angles are preserved by scaling, we may consider the rectangular grid $\mathcal{G}_m = \{ 0, 1, \dots, m \} \times \{ 0, 1, \dots, m \}$ instead of $\cT_h$ in $\omega_i$.
Take any three noncollinear nodes $p^1$, $p^2$, $p^3$ in $\mathcal{G}_m$.
Recall that the sine of the angle $\angle p^1 p^2 p^3$ formed by these nodes is given by
\begin{equation*}
\sin \angle p^1 p^2 p^3 = \frac{| \det [ p^1 - p^2, p^2 - p^3] |}{| p^1 - p^2 | | p^2 - p^3 | }.
\end{equation*}
Since $\det [p^1 - p^2, p^2 - p^3 ]$ is a nonzero integer, we have $| \det [p^1 - p^2, p^2 - p^3 ] | \geq 1$.
Hence, we get
\begin{equation}
\label{Prop1:angle}
\sin \angle p^1 p^2 p^3 \geq \frac{1}{\sqrt{2}m \cdot \sqrt{2}m} = \frac{1}{2m^2}.
\end{equation}
Finally, we see that taking $p^1 = (m-1, m)$, $p^2 = (0, 0)$, and $p^3 = (m-2, m-1)$ yields $| \det [p^1 - p^2, p^2 - p^3 ] | = 1$, $| p^1 - p^2 | \approx m$, and $| p^2 - p^3 | \approx m$.
This verifies that the inequality~\eqref{Prop1:angle} is sharp.
\end{proof}

The following lemma states that any $u \in \tS_h$ can be locally approximated by a linear function $J_i u \in \mathbb{P}_1 ( \overline{\omega}_i )$ that preserves positivity.
The proof of Lemma~\ref{Lem:J_i} involves intricate arguments and is presented in Section~\ref{Sec:J_i}.

\begin{lemma}
\label{Lem:J_i}
Suppose that Assumption~\ref{Ass:cT} holds.
For any $v \in \tS_h$ and a region $\omega_i$, $i \in \cI_H$, there exists a linear function $J_i v \in \mathbb{P}_1 (\overline{\omega}_i)$ such that
\begin{equation}
    \label{Lem1:J_i}
    \begin{cases}
    \displaystyle 0 \leq J_i v \leq v, & \quad \text{ if } v > 0, \\
    \displaystyle v \leq J_i v \leq 0, & \quad \text{ if } v < 0, \\
    \displaystyle J_i v = 0, & \quad \text{ if } v = 0, \\
    \end{cases}
    \quad \text{at the } \cT_h\text{-vertices in } \overline{\omega}_i,
\end{equation}
and
\begin{multline}
    \label{Lem2:J_i}
    \| v - J_i v \|_{L^2 (\omega_i)} \lesssim
    H | v - J_i v |_{H^1 (\omega_i)} + H^2 | v |_{H^2 (\omega_i)} \\
    \lesssim \frac{H^2}{\sin \alpha (H,h)} \left( 1 + \log \frac{H}{h} \right)^{\frac{1}{2}} | v |_{H^2 (\omega_i)}.
\end{multline}
Moreover, for $\Gamma \subset \partial \omega_i$ with nonvanishing measure, $J_i v$ satisfies the following:
\begin{equation}
\label{J_i_BC}
\begin{split}
&\text{if } v=0 \quad\quad\quad\text{on } \Gamma , \quad \text{ then } J_i v = 0 \text{ on } \Gamma, \\
&\text{if } v = \frac{\partial v}{\partial \nu} = 0 \text { on } \Gamma , \quad \text{ then } J_i v = 0 \quad\text { in } \omega_i,
\end{split}
\end{equation}
where $\nu$ is the outward unit normal vector field along $\Gamma$.
\end{lemma}

\begin{remark}
At first glance, Lemma~\ref{Lem:J_i} may seem counterintuitive, as it appears unlikely that interpolation by a linear function can both preserve the sign of a function at many nodal points in $\overline{\omega}_i$ and achieve good approximation properties.
However, we provide the following intuition for Lemma~\ref{Lem:J_i}.

Given a function $v \in \widetilde{S}_h$, if the direction of the gradient field $\nabla v$ is nearly constant in $\omega_i$, i.e., $v$ is $\alpha$-biased on $\omega_i$ in the sense of Definition~\ref{Def:biased}, then $v$ remains almost constant along some direction and varies predominantly along the perpendicular direction.
In this case, $v$ behaves similarly to a one-dimensional function along the perpendicular direction, and we can construct an appropriate linear function that preserves positivity while also providing desired approximation properties, using an one-dimensional argument; see Lemma~\ref{Lem:linear}.
On the other hand, if $v$ not $\alpha$-biased, then there exist two directions with a sufficiently large angle such that the partial derivatives of $v$ vanish at certain points.
In this case, it is enough to take a constant function that touches $v$ at an extremal point of $|v|$; this constant function preserves the positivity of $v$ and ensures the desired approximation properties, thanks to a discrete Poincar\'{e}--Friedrichs-type inequality~(see Lemma~\ref{Lem:PF}).

The proof of Lemma~\ref{Lem:J_i} provided in Section~\ref{Sec:J_i} formalizes this intuition through rigorous mathematical arguments.
\end{remark}

For $v \in \tS_h$, we define $J_H v \in V_0$ as
\begin{equation}
\label{J_H}
    J_H v = \sum_{i \in \cI_H} (J_i v) \phi_i,
\end{equation}
where $J_i$ and $\phi_i$ were given in Lemma~\ref{Lem:J_i} and~\eqref{pou_basis}, respectively.
Note that, thanks to~\eqref{J_i_BC}, $J_H u$ satisfies the boundary condition of $S$.
Approximation properties of the operator $J_H \colon \tS_h \rightarrow V_0$ are summarized in Lemma~\ref{Lem:J_H}.

\begin{lemma}
\label{Lem:J_H}
Suppose that Assumption~\ref{Ass:cT} holds.
The operator $J_H \colon \tS_h \rightarrow V_0$ defined in~\eqref{J_H} satisfies
\begin{equation}
\label{Lem1:J_H}
    \begin{cases}
    \displaystyle 0 \leq J_H v \leq v, & \quad \text{ if } v > 0, \\
    \displaystyle v \leq J_H v \leq 0, & \quad \text{ if } v < 0, \\
    \displaystyle J_H v = 0, & \quad \text{ if } v = 0, \\
    \end{cases}
    \quad \text{at the vertices of } \cT_h,
\end{equation}
and
\begin{multline}
\label{Lem2:J_H}
    \| v - J_H v \|_{L^2 (\Omega)} + H | v - J_H v |_{H^1 (\Omega)}
    + H^2 | J_H v |_{H^2 (\Omega)} \\
    \lesssim \frac{H^2}{\sin \alpha (H,h)} \left( 1 + \log \frac{H}{h} \right)^{\frac{1}{2}} | v |_{H^2 (\Omega)}, \quad v \in \tS_h.
\end{multline}
\end{lemma}
\begin{proof}
Since~\eqref{Lem1:J_H} is a  direct consequence of~\eqref{pou_basis0},~\eqref{pou_basis2}, and~\eqref{Lem1:J_i}, we provide a proof of~\eqref{Lem2:J_H} only.
Take any $T \in \cT_H$.
Let $\{ x^i \}_{i=1}^{m_T}$ be the vertices of $T$.
Note that $m_T$ is uniformly bounded with respect to $T$ because $\cT_H$ is quasi-uniform.
We define a region $\omega_T \subset \Omega$ as
\begin{equation*}
    \overline{\omega}_T = \bigcup_{i=1}^{m_T} \overline{\omega}_i,
\end{equation*}
where the definition of $\omega_i$ was given in~\eqref{omega_i}.
For $1 \leq i \leq m_T$ and $0 \leq j \leq 2$, we have
\begin{equation} \begin{split}
\label{Lem3:J_H}
| ( v - J_i v ) \phi_i |_{H^j (\omega_i)}
&\lesssim \sum_{l=0}^j | v - J_i v |_{H^l (\omega_i)} \| \nabla^{j-l} \phi_i \|_{L^{\infty} (\omega_i)} \\
&\lesssim \frac{H^{2-j}}{\sin \alpha (H,h)} \left( 1 + \log \frac{H}{h} \right)^{\frac{1}{2}} | v |_{H^2 (\omega_i)},
\end{split} \end{equation}
where the last inequality is due to~\eqref{pou_basis3} and~\eqref{Lem2:J_i}.
It follows that
\begin{multline}
\label{Lem4:J_H}
| v - J_H v |_{H^j (T)} \stackrel{\eqref{pou_basis2}}{=} \left| \sum_{i=1}^{m_T} (v - J_i v) \phi_i \right|_{H^j (T)} \\
\leq \sum_{i=1}^{m_T} | (v - J_i v ) \phi_i |_{H^j (\omega_i)}
\stackrel{\eqref{Lem3:J_H}}{\lesssim} \frac{ H^{2-j} }{\sin \alpha (H,h)} \left( 1 + \log \frac{H}{h} \right)^{\frac{1}{2}} | v |_{H^2 (\omega_T)}.
\end{multline}
Summing~\eqref{Lem4:J_H} over all $T$ yields
\begin{equation}
\label{Lem5:J_H}
    | v - J_H v |_{H^j (\Omega)} \lesssim \frac{H^{2-j}}{\sin \alpha (H,h)} \left( 1 + \log \frac{H}{h} \right)^{\frac{1}{2}} | v |_{H^2 (\Omega)}.
\end{equation}
Finally, combining~\eqref{Lem5:J_H} and
\begin{equation*}
\begin{split}
    | J_H v |_{H^2 (\Omega)} &\leq | v - J_H v |_{H^2 (\Omega)} + | v |_{H^2 (\Omega)} \\
    &\stackrel{\eqref{Lem5:J_H}}{\lesssim} \frac{1}{\sin \alpha (H,h)} \left( 1 + \log \frac{H}{h} \right)^{\frac{1}{2}} | v |_{H^2 (\Omega)}
\end{split}
\end{equation*}
yields~\eqref{Lem2:J_H}.
\end{proof}

Finally, we are ready to estimate the constant $C_0$ in Assumption~\ref{Ass:stable} for the two-level method.
Using the approximation properties of $J_H$ presented in Lemma~\ref{Lem:J_H}, we deduce the following theorem.

\begin{theorem}
\label{Thm:2L}
Suppose that the following conditions hold:
\begin{itemize}
\item The finite element discretization~\eqref{model_FEM} satisfies Assumption~\ref{Ass:FEM}.
\item The local spaces~\eqref{V_k} satisfy Assumption~\ref{Ass:local}.
\item The coarse space~\eqref{V_0} satisfies Assumption~\ref{Ass:coarse}.
\item The triangulations $\cT_h$ and $\cT_H$ satisfy Assumption~\ref{Ass:cT}.
\end{itemize}
Then, in the two-level decomposition~\eqref{2L}, Assumption~\ref{Ass:stable} is satisfied with
\begin{equation*}
C_0^2 \approx \frac{1}{\sin^2 \alpha (H,h)} \left( 1 + \log \frac{H}{h} \right) \left( 1+  \left(\frac{H}{\delta} \right)^3 \right).
\end{equation*}
\end{theorem}
\begin{proof}
Throughout this proof, we write
\begin{equation*}
C(H,h) = \frac{1}{\sin^2 \alpha (H,h)} \left( 1 + \log \frac{H}{h} \right).
\end{equation*}
Take any $v,w \in V$ with $v, v+w \in K_h$.
We set $w_0 = J_H (E_h w)$, where $J_H$ was given in~\eqref{J_H}.
Thanks to~\eqref{E_h_preserving} and~\eqref{Lem1:J_H}, we have $v + R_0^* w_0 \in K_h$.
Moreover, we observe that
\begin{equation}
    \label{Thm1:2L}
    | J_H (E_h w) |_{H^2 (\Omega)}^2
    \stackrel{\eqref{Lem2:J_H}}{\lesssim} C(H,h) | E_h w |_{H^2 (\Omega)}^2
    \stackrel{\eqref{E_h}}{\lesssim} C(H,h) \| w \|_h^2.
\end{equation}
Then we can estimate $a_h (R_0^* w_0, R_0^* w_0)$ as follows:
\begin{equation}
    \label{Thm2:2L}
    a_h (R_0^* w_0, R_0^* w_0)
    \stackrel{\eqref{norm}}{\approx} \| R_0^* J_H (E_h w) \|_h^2
    \stackrel{\eqref{intergrid}}{\lesssim} | J_H (E_h w) |_{H^2 (\Omega)}^2
    \stackrel{\eqref{Thm1:2L}}{\lesssim} C(H,h) \| w \|_h^2.
\end{equation}

Meanwhile, by Assumption~\ref{Ass:local}, there exist $w_k \in V_k$, $1 \leq k \leq N$, such that $w - R_0^* w_0 = \sum_{k=1}^N R_k^* w_k$, $v + R_k^* w_k \in K_h$, and
\begin{equation} \begin{split}
    \label{Thm3:2L}
    \sumk a_h &(\Rw, \Rw) \approx \sumk \| \Rw \|_h^2 \\
    &\lesssim \| w - R_0^* w_0 \|_h^2 + \frac{1}{\delta^2} \sumT | w - R_0^* w_0 |_{H^1 (T)}^2 \\
    &\quad + \frac{1}{H \delta^3} \| E_h (w - R_0^* w_0) \|_{L^2 (\Omega)}^2 + \left( \frac{H}{\delta} \right)^3 | E_h (w - R_0^* w_0) |_{H^2 (\Omega)}^2,
\end{split} \end{equation}
where the $\approx$-relation in the first line is due to~\eqref{norm}.
In order to estimate the rightmost-hand side of~\eqref{Thm3:2L}, we estimate some norms of $w - R_0^* w_0$; using~\eqref{E_h},~\eqref{Lem2:J_H},~\eqref{intergrid}, and~\eqref{Thm1:2L}, we get
\begin{subequations}
\label{Thm4:2L}
\begin{equation} \begin{split}
\| w - R_0^* w_0 \|_{L^2 (\Omega)}^2
&\lesssim \| w - E_h w \|_{L^2 (\Omega)}^2 + \| E_h w - J_H (E_h w) \|_{L^2 (\Omega)}^2  \\
&\quad + \| J_H (E_h w) - R_0^* J_H (E_h w) \|_{L^2 (\Omega)}^2 \\
&\lesssim \left( h^4 + C (H,h) H^4 \right) \| w \|_h^2
\lesssim C (H,h) H^4 \| w \|_h^2.
\end{split} \end{equation}
In the same manner, we have
\begin{equation} \begin{split}
    &\sumT | w - R_0^* w_0 |_{H^1 (T)}^2  \\
    &\lesssim \sumT  | w - E_h w |_{H^1 (T)}^2 + | E_h w - J_H (E_h w) |_{H^1 (\Omega)}^2 \\
    &\quad + \sumT | J_H (E_h w) - R_0^* J_H (E_h w) |_{H^1 (T)}^2 \\
    &\lesssim \left( h^2 + C (H,h) H^2 \right) \| w \|_h^2
    \lesssim C (H,h) H^2 \| w \|_h^2
\end{split} \end{equation}
and
\begin{equation}
\| w - R_0^* w_0 \|_h^2
\lesssim \|w \|_h^2 + \| R_0^* w_0 \|_h^2
\lesssim \|w \|_h^2 + | J_H (E_h w) |_{H^2 (\Omega)}^2
\lesssim C (H,h) \| w \|_h^2.
\end{equation}
\end{subequations}
Then, invoking~\eqref{Thm3:2L} with~\eqref{E_h} and~\eqref{Thm4:2L} yields
\begin{equation}
    \label{Thm5:2L}
    \sumk a (R_k^* w_k, R_k^* w_k ) \lesssim C (H,h) \left[ 1 + \left( \frac{H}{\delta} \right)^2 + \left( \frac{H}{\delta} \right)^3 \right] \| w \|_h^2.
\end{equation}
Combining~\eqref{Thm2:2L} and~\eqref{Thm5:2L} yields the desired result.
\end{proof}

Theorem~\ref{Thm:2L} implies that the two-level method is scalable because the convergence rate depends only on $H/h$ and $H/\delta$.
In both cases of small overlap $\delta \approx h$ and generous overlap $\delta \approx H$, the convergence rate is uniformly bounded by a function of $H/h$.
Therefore, even when the fine mesh size $h$ is very small, a large number of iterations is not necessary to achieve a solution of~\eqref{model_FEM} with a certain level of accuracy, provided that we have a sufficient number of subdomains such that $H/h$ remains fixed.

\begin{remark}
\label{Rem:conforming}
In the two-level method, we employ a conforming coarse space $V_0$, while the choice of the fine finite element space $V$ can be any discretization that fits within the framework outlined in Section~\ref{Sec:Pre}~\cite{Lee:1993}.
Should one opt for a coarse space defined in terms of a nonconforming or interior penalty method, it is essential to design an appropriate intergrid transfer operator that maps $V_0$ to $V$; refer to~\cite[Lemma~6.1]{Brenner:1996} and~\cite[Lemma~3.3]{BW:2005} for nonconforming and interior penalty methods, respectively.
However, the intergrid transfer operators discussed in these references are defined using certain coarse enriching operators, which do not preserve nodal values on the fine grid.
Consequently, they cannot be employed to construct a stable decomposition that satisfies the pointwise constraints $v + R_0^* w_0 \in K_h$ as described in Assumption~\ref{Ass:stable}.
Hence, we use a common conforming coarse space for all types of fine discretizations.
\end{remark}

\section{Local and coarse problems}
\label{Sec:Solvers}
When implementing Algorithm~\ref{Alg:ASM}, we have to consider how to solve local and coarse problems of the form 
\begin{equation}
\label{local_original}
\min_{\substack{w_k \in V_k, \\ v + \Rw \in K_h}} F_h \left( v + \Rw \right),
\end{equation}
where $v \in S_h$ and $0 \leq k \leq N$.
In this section, we address how to solve problems of the form~\eqref{local_original}.
With an abuse of notation, we do not distinguish between finite element functions and the corresponding vectors of degrees of freedom.

Let $A_h$ and $f_h$ be the stiffness matrix and the load vector induced by $a_h(\cdot, \cdot)$ and $(f, \cdot)$ in~\eqref{model_FEM}, respectively, i.e.,
\begin{equation*}
a_h (v, w) = v^{\rT} A_h w, \quad
(f, v) = f_h^{\rT} v,
\end{equation*}
for $v, w \in S_h$.
In addition, let $R_k^{\rT}$ be the matrix representation of the operator $R_k^*$.
Then~\eqref{local_original} is rewritten as the following constrained quadratic optimization problem:
\begin{equation}
\label{local_quad}
\min_{w_k \in K_k} \left\{ \frac{1}{2} w_k^{\rT} A_k w_k - f_k^{\rT} w_k \right\},
\end{equation}
where $A_k = R_k A_h R_k^{\rT}$, $f_k = R_k (A_h v - f_h)$, and
\begin{equation*}
K_k = \left\{ w_k \in V_k : v + R_k^{\rT} w_k \in K_h \right\}.
\end{equation*}

\subsection{Local problems}
In local problems, i.e., when $1 \leq k \leq N$, the operator $R_k^{\rT}$ is the natural extension operator, which implies that the set $K_k$ encodes pointwise constraints as in the full-dimensional problem~\eqref{model_FEM}.
Hence, any numerical method for the full-dimensional problem~\eqref{model_FEM} can be utilized to solve local problems.
For instance, one may adopt the primal-dual active set method~\cite{BIK:1999,HIK:2002} to solve~\eqref{local_quad}.
We note that an auxiliary linear problem~\cite{BDS:2018,BSW:2022,Lee:2013} should be solved at each iteration of the primal-dual active set method.
Alternatively, as the Euclidean projection onto the set $K_k$ of pointwise constraints admits a closed formula~\cite{LP:2021}, any modern forward-backward splitting algorithms for convex optimization~(see, e.g.,~\cite{BT:2009,CC:2019,OC:2015}) can be applied to solve~\eqref{local_quad}.

\subsection{Coarse problems}
In coarse problems, the operator $R_0^{\rT}$ corresponds to the coarse prolongation matrix.
Unlike in local problems, the constraint set $K_0$ is not given pointwise; the $K_h$-constraints are imposed at the vertices of the fine mesh $\cT_h$, while the degrees of freedom of $V_0$ are located at the vertices of the coarse mesh $\cT_H$.
To tackle such problems, interior point methods capable of handling general matrix inequality constraints can be employed~\cite{AG:1999,VC:1993}.
On the other hand, in~\cite{BTW:2003}, a nonlinear Gauss--Seidel algorithm was proposed to solve coarse problems that arise in Schwarz methods for second-order variational inequalities; this algorithm iteratively minimizes the energy functional with respect to each coarse degree of freedom.

For simpler algorithms, an alternative approach based on Fenchel--Rockafellar duality may be considered~\cite{LP:2020,LP:2021}.
We begin by rewriting the constraint set $K_k$ in \eqref{local_quad} as follows:
\begin{equation*}
\label{coarse}
K_k = \left\{ w_k \in V_k : J R_k^{\rT} w_k \leq \psi - J v \right\},
\end{equation*}
where $J$ is a boolean matrix, whose entries are 0's and 1's, that extracts the degrees of freedom corresponding to function values from the full set of degrees of freedom, which includes both function and derivative values.
The following proposition states that a solution of the problem~\eqref{local_quad} can be recovered from a solution of a particular dual problem.

\begin{proposition}
\label{Prop:dual}
Let $w_k$ be a solution of the problem~\eqref{local_quad} and let $\lambda_k$ be a solution of the constrained quadratic optimization problem
\begin{equation}
\label{coarse_dual}
\min_{\lambda_k \geq 0} \left\{ \frac{1}{2} (R_k J^{\rT} \lambda_k - f_k)^{\rT} A_k^{-1} (R_k J^{\rT} \lambda_k - f_k) + (\psi - J v)^{\rT} \lambda_k \right\}. 
\end{equation}
Then we have
\begin{equation*}
w_k = A_k^{-1} (f_k - R_k J^{\rT} \lambda_k ).
\end{equation*}
\end{proposition}
\begin{proof}
It is straightforward by the theory of Fenchel--Rockafellar duality~\cite[Proposition~4.1]{LP:2020}.
One may refer to~\cite[Section~2.3]{LP:2021} for a concrete example for the construction of a dual problem.
\end{proof}

Thanks to Proposition~\ref{Prop:dual}, one may solve the dual problem~\eqref{coarse_dual} instead of the original problem~\eqref{local_quad} in order to obtain the solution of~\eqref{local_quad}.
Since computing the Euclidean projection onto the constraint set $\lambda_k \geq 0$ is straightforward,~\eqref{coarse_dual} can be solved using, e.g., a forward-backward splitting algorithm~\cite{BT:2009,CC:2019,OC:2015}.

\section{Numerical experiments}
\label{Sec:Numerical}
In this section, we conduct numerical experiments that  support the theoretical results presented in this paper.
All codes used in our numerical experiments were programmed using MATLAB R2022b and performed on a desktop equipped with AMD Ryzen~5 5600X CPU~(3.7GHz, 6C), 40GB RAM, and the operating system Windows 10~Pro.

Let $\Omega = (0,1)^2 \subset \mathbb{R}^2$, and let $\cT_H$ be a coarse triangulation of $\Omega$ consisting of $N = 1/H \times 1/ H$ square elements.
We further refine $\cT_H$ to obtain a fine triangulation $\cT_h$, so that $\cT_h$ consists of $1/h \times 1/h$ square elements.
By enlarging each coarse element in $\cT_H$ to include its surrounding layers of fine elements in $\cT_h$ with width $\delta$ such that $0 < \delta < H/2$, we construct an overlapping domain decomposition $\{ \Omega_k \}_{k=1}^N$ of $\Omega$.

In~\eqref{model_FEM}, we choose the finite element space $S_h$ as the Bogner--Fox--Schmit bicubic element space~\cite{BFS:1965,Valdman:2020} defined on $\cT_h$.
Since $S_h$ is conforming, we can simply set the discrete bilinear form $a_h (\cdot, \cdot)$ such that it agrees with the continuous bilinear form $a( \cdot, \cdot)$.
Moreover, by setting $\| v \|_h = \sqrt{a_h (v, v)}$ for $v \in S_h$, $\tS_h = S_h$, and $E_h$ as the identity operator, it is evident that Assumption~\ref{Ass:FEM} is satisfied.
On the other hand, Assumption~\ref{Ass:local} can be verified by combining the result in~\cite{Zhang:1996} with a trace theorem-type argument~\cite{DW:1994}.
Setting $\phi_i$ in~\eqref{pou_basis} in terms of the Bogner--Fox--Schmit basis function on $\cT_H$, and choosing the coarse prolongation operator $R_0^T$ as the natural nodal interpolation operator guarantees that Assumption~\ref{Ass:coarse} holds, invoking standard interpolation error estimates~\cite{BS:2008}.
Finally, Proposition~\ref{Prop:angle} implies that Assumption~\ref{Ass:cT} holds.
In summary, we have provided a setting that satisfies all the assumptions given in Theorem~\ref{Thm:2L}.

In Algorithm~\ref{Alg:ASM}, the step size $\tau$ is given by $\tau = 1/5$~(cf.~\cite[equation~(5.6)]{Park:2020}).
All local problems are solved using the primal-dual active set method~\cite{BIK:1999,HIK:2002}, with the stopping criterion given by
\begin{equation*}
\frac{\| w_k^{(n+1)} - w_k^{(n)} \|_{A_k}} {\|w_k^{(n)} \|_{A_k}} < 10^{-12}.
\end{equation*}
In the two-level method, coarse problems are solved using the interior point method~\cite{AG:1999,VC:1993}, as implemented in the \texttt{quadprog} subroutine of the MATLAB Optimization Toolbox~\cite{MATLAB}.

In each example, we use a reference solution $u_h$ to measure the convergence rates of the proposed methods.
The reference solution is obtained by sufficiently many iterations of the primal-dual active set method.

\subsection{Displacement obstacle problem of clamped plates}

\begin{figure}
\centering
\subfloat[][One-level, $\delta/h = 2^0$]{ \includegraphics[width=0.31\linewidth]{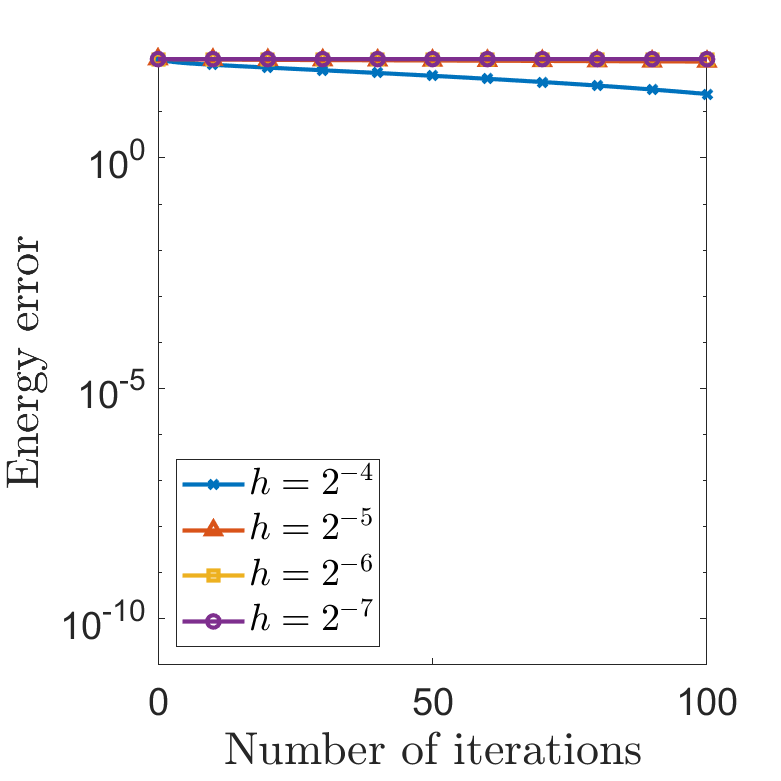} }
\subfloat[][One-level, $\delta/h = 2^1$]{ \includegraphics[width=0.31\linewidth]{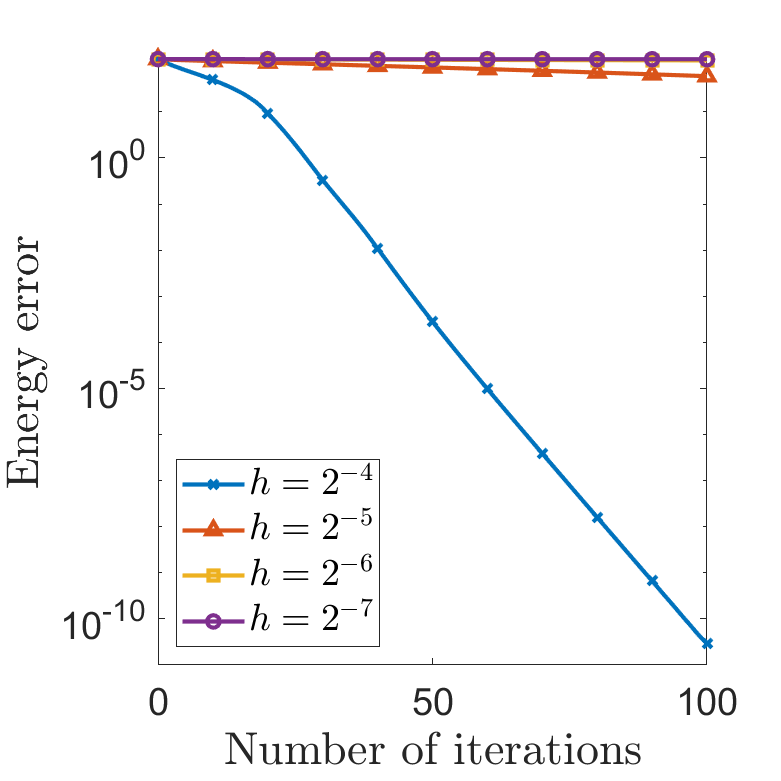} }
\subfloat[][One-level, $\delta/h = 2^2$]{ \includegraphics[width=0.31\linewidth]{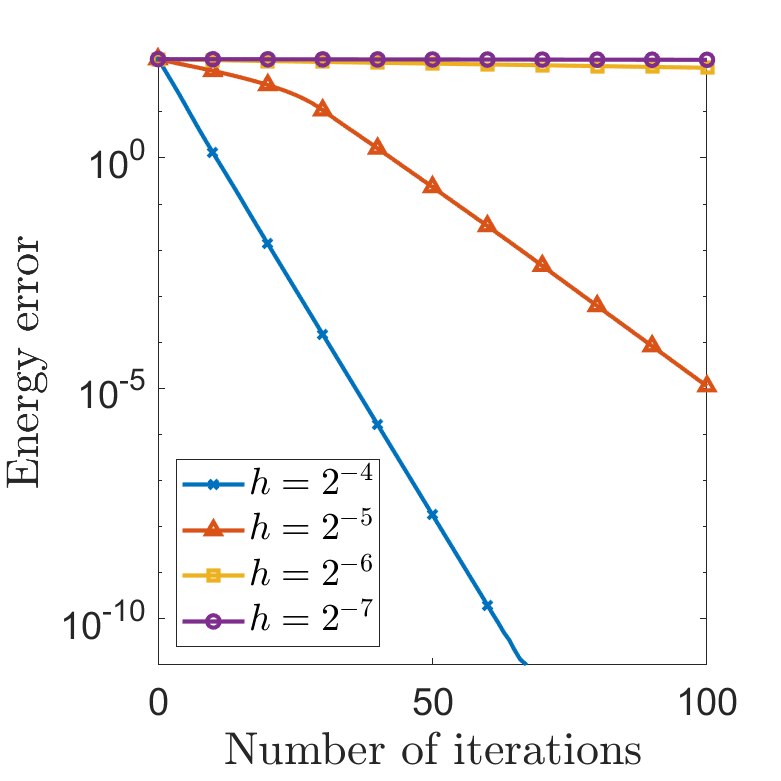} } \\
\subfloat[][Two-level, $\delta/h = 2^0$]{ \includegraphics[width=0.31\linewidth]{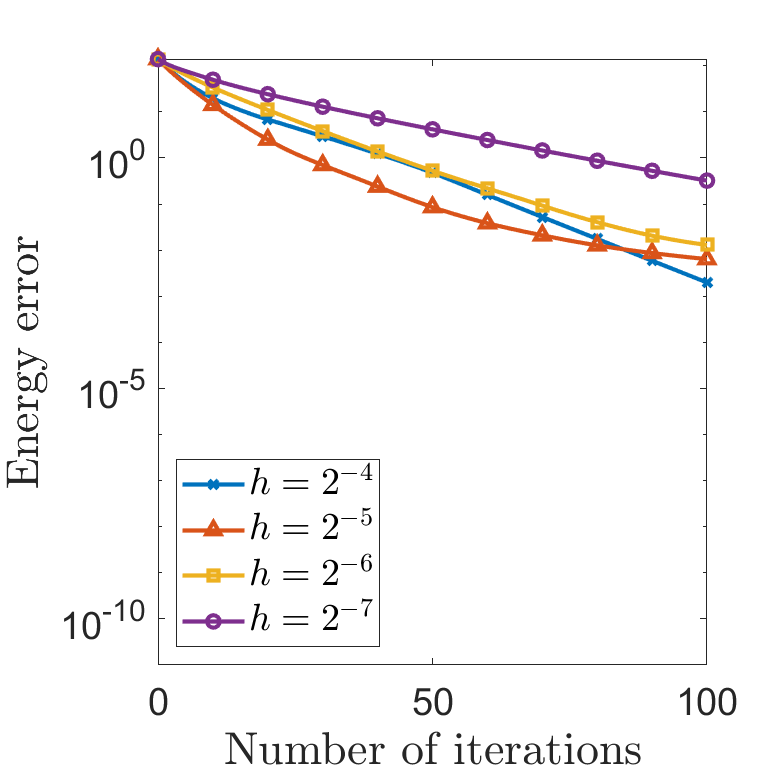} }
\subfloat[][Two-level, $\delta/h = 2^1$]{ \includegraphics[width=0.31\linewidth]{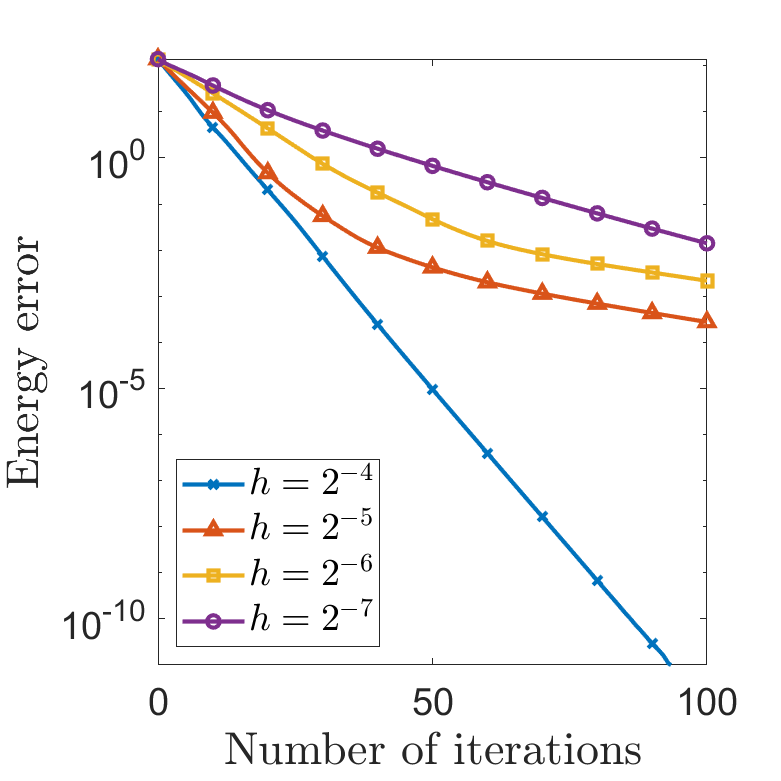} }
\subfloat[][Two-level, $\delta/h = 2^2$]{ \includegraphics[width=0.31\linewidth]{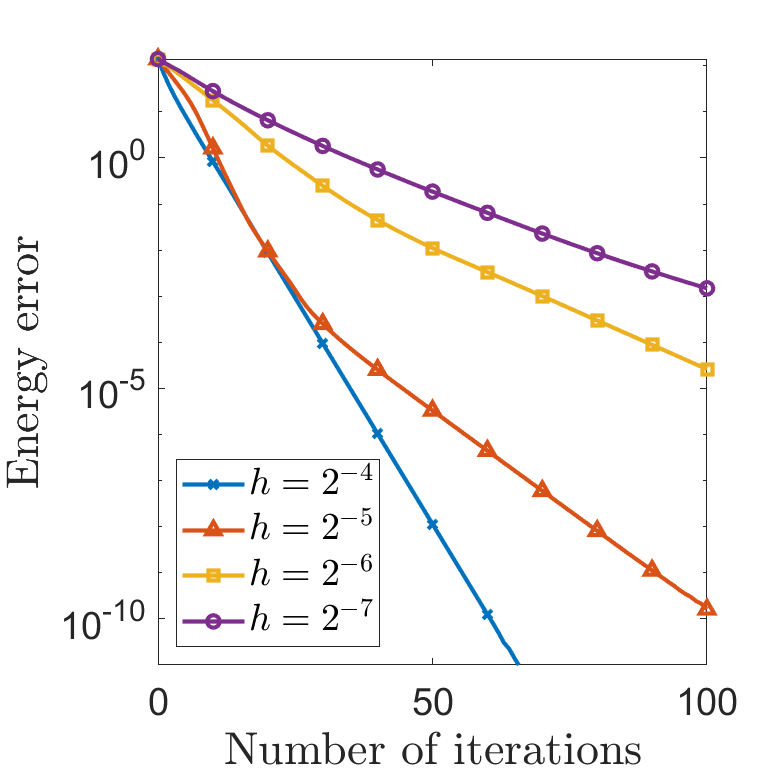} }
\caption{Decay of the relative energy error $\frac{F_h (\un) - F_h(u_h)}{| F_h (u_h) |}$ in the one- and two-level additive Schwarz methods for the displacement obstacle problem of clamped plates~\eqref{PL}.
The subdomain size $H$ and the mesh size $h$ vary satisfying $H/h = 2^3$.}
\label{Fig:PL}
\end{figure}

First, we consider the displacement obstacle problem of clamped plates~\eqref{PL}.
In~\eqref{model_opt}, let $f = 10^3$ and
\begin{equation*}
\psi (x_1, x_2) = \frac{1}{2} - \left( x_1 - \frac{1}{2} \right)^2 - \left( x_2 - \frac{1}{2} \right)^2.
\end{equation*}
The convergence curves of the one-level additive Schwarz method are depicted in Figure~\ref{Fig:PL}(a--c).
While the convergence rates of the one- and two-level methods are comparable for $h = 2^{-4}$, where the dimension of the coarse space $V_0$ is relatively small and does not significantly affect the convergence rate, the difference becomes evident as the mesh size $h$ decreases, increasing the dimension of the coarse space.
The convergence rate of the one-level method accelerates with increasing overlap width $\delta$ but slows down significantly as the mesh size $h$ decreases.
This behavior aligns with Theorem~\ref{Thm:1L}; the convergence rate of the one-level method depends on $\delta^{-4}$ when $H/\delta$ is fixed.
On the contrary, as shown in Figure~\ref{Fig:PL}(d--f), the asymptotic convergence rate of the two-level method does not decrease even if $h$ becomes small.
More precisely, for $h \leq 2^{-5}$, the convergence curves appear nearly parallel to each other with sufficiently large iterations.
This observation confirms the assertion of Theorem~\ref{Thm:2L}, which states that the convergence rate of the two-level method remains uniformly bounded given fixed ratios of $H/h$ and $H/\delta$.

\subsection{Elliptic distributed optimal control problem}

\begin{figure}
\centering
\subfloat[][One-level, $\delta/h = 2^0$]{ \includegraphics[width=0.31\linewidth]{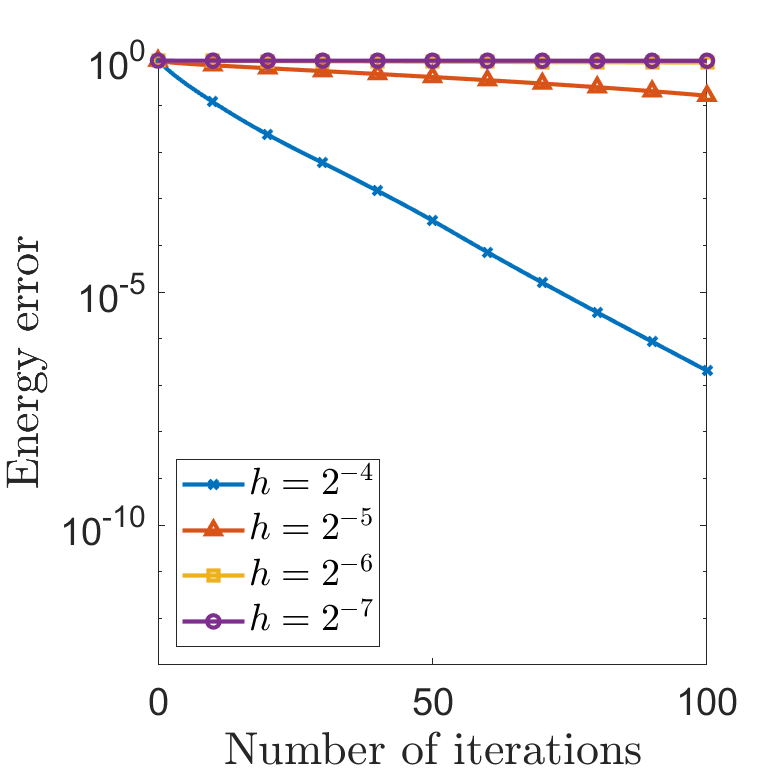} }
\subfloat[][One-level, $\delta/h = 2^1$]{ \includegraphics[width=0.31\linewidth]{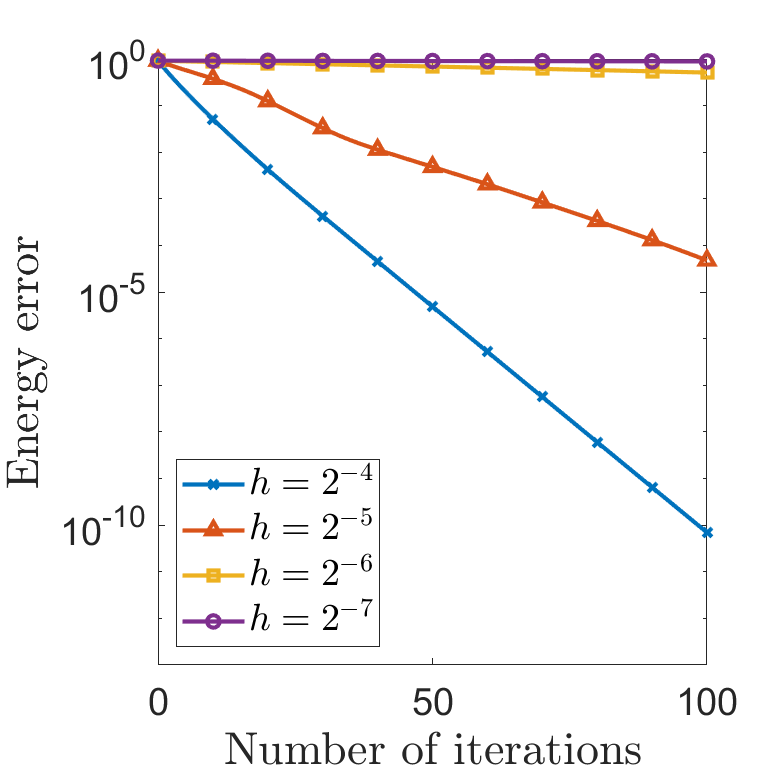} }
\subfloat[][One-level, $\delta/h = 2^2$]{ \includegraphics[width=0.31\linewidth]{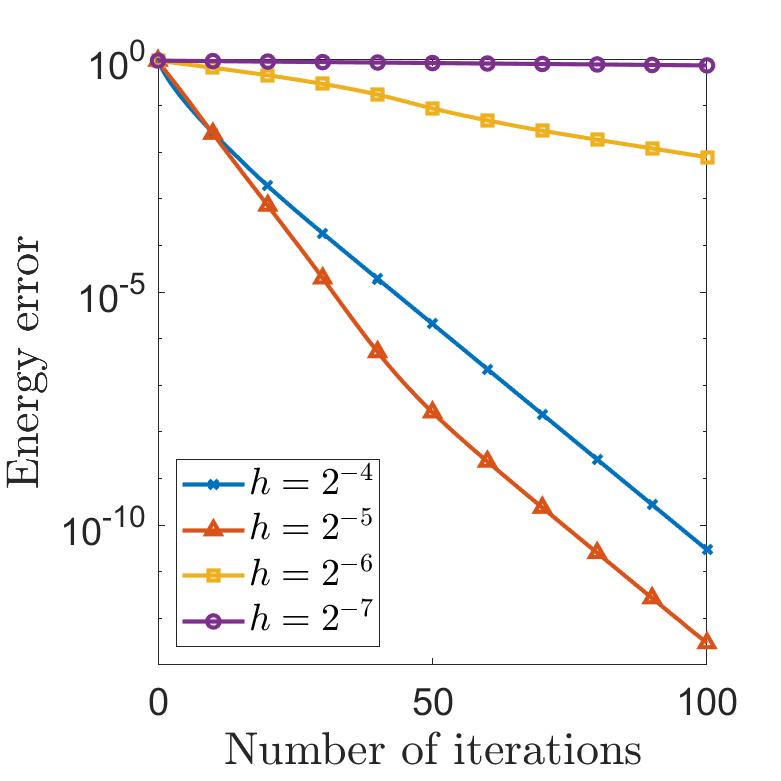} } \\
\subfloat[][Two-level, $\delta/h = 2^0$]{ \includegraphics[width=0.31\linewidth]{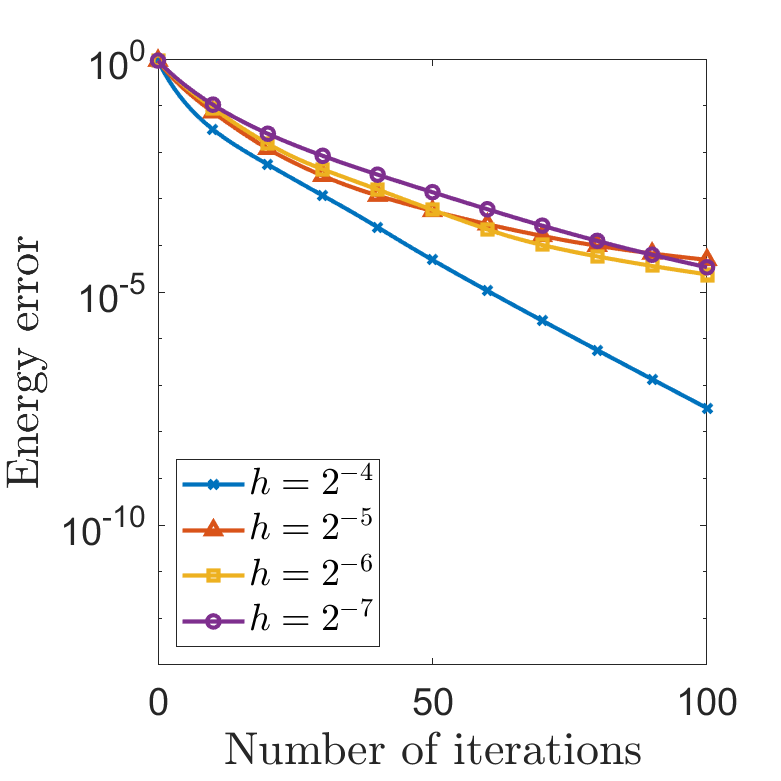} }
\subfloat[][Two-level, $\delta/h = 2^1$]{ \includegraphics[width=0.31\linewidth]{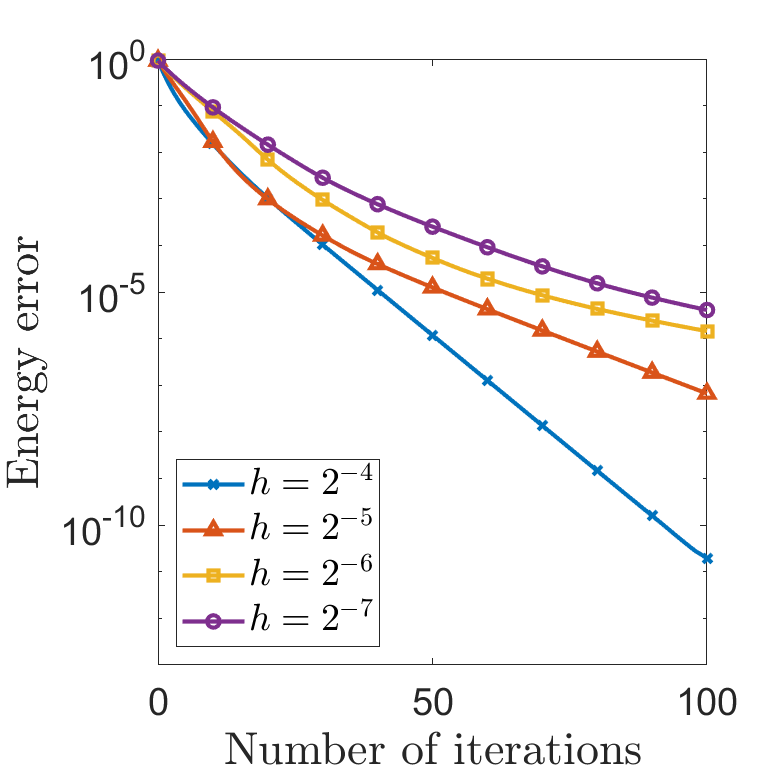} }
\subfloat[][Two-level, $\delta/h = 2^2$]{ \includegraphics[width=0.31\linewidth]{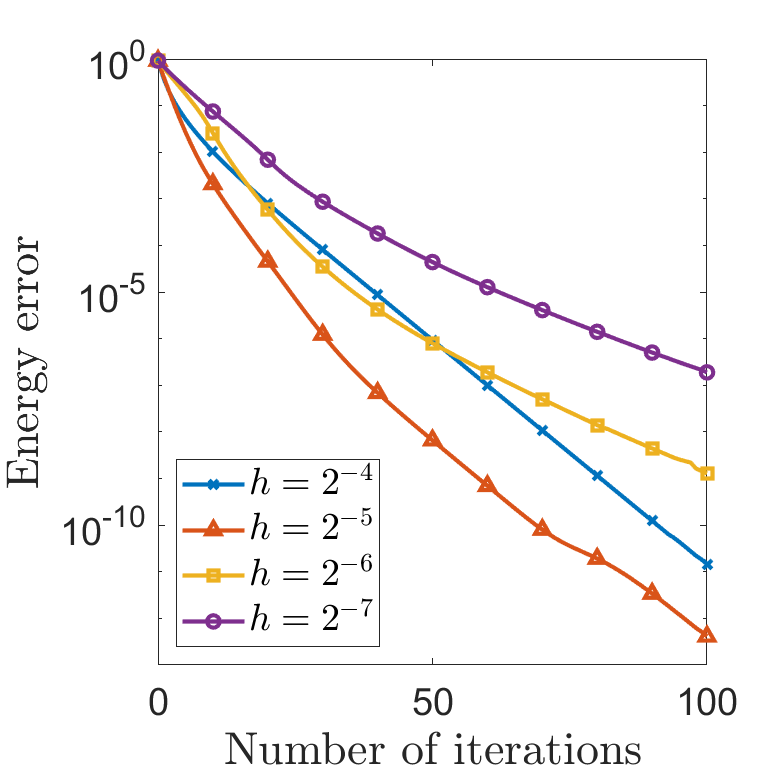} }
\caption{Decay of the relative energy error $\frac{F_h (\un) - F_h(u_h)}{| F_h (u_h) |}$ in the one- and two-level additive Schwarz methods for the elliptic distributed optimal control problem~\eqref{OC}.
The subdomain size $H$ and the mesh size $h$ vary satisfying $H/h = 2^3$.}
\label{Fig:OC}
\end{figure}

Next, we deal with the elliptic distributed optimal control problem~\eqref{OC}.
We consider an example appeared in~\cite{BDS:2014,BST:2021,LGY:2009}; we set $\beta = 10^{-4}$, $f(x_1, x_2) = \sin ( 4\pi x_1 x_2 ) + 1.5$, and $\psi = 1$.
Figure~\ref{Fig:OC} plots the relative energy error $\frac{F_h (\un) - F_h(u_h)}{| F_h (u_h) |}$  in both the one- and two-level additive Schwarz methods, with varying $H$, $h$, and $\delta$ satisfying $H/h = 2^3$.
The convergence behavior of these methods mirrors that of the displacement obstacle problem for clamped plates.
Specifically, while the convergence rate of the one-level method sharply declines as $h$ decreases, the two-level method exhibits rapid decay of the energy error even with larger values of $h$.
This verifies the effectiveness of the proposed coarse space in accelerating the convergence of additive Schwarz methods.

\section{Proof of Lemma~\ref{Lem:J_i}}
\label{Sec:J_i}
This section is devoted to the proof of Lemma~\ref{Lem:J_i}.
We begin by introducing some technical tools. Lemma~\ref{Lem:H1} presents elementary vector inequalities.

\begin{lemma}
\label{Lem:H1}
Let $e_1$ and $e_2$ be nonparallel unit vectors in $\mathbb{R}^2$, and let $\alpha$ be the angle formed by $e_1$ and $e_2$.
Then, for $p \in \mathbb{R}^2$, we have
\begin{equation*}
\frac{1}{2} \left( ( p \cdot e_1 )^2 + ( p \cdot e_2 )^2 \right)
\leq
|  p |^2
\leq \frac{3}{\sin^2 \alpha} \left( ( p \cdot e_1 )^2 + ( p \cdot e_2 )^2 \right).
\end{equation*}
\end{lemma}
\begin{proof}
It is elementary.
\end{proof}

Next, we present a discrete Poincar\'{e}--Friedrichs-type inequality for $C^1$ conforming finite element functions.

\begin{lemma}
\label{Lem:PF}
For $v \in \tS_h$ and a region $\omega_i$, $i \in \cI_H$,  suppose that there exists a point $p^0 \in \overline{\omega}_i$ such that $v (p^0) = 0$.
In addition, suppose that there exist nonparallel unit vectors $e_1, e_2 \in \mathbb{R}^2$ such that $\nabla v (p^1) \cdot e_1 = \nabla v (p^2) \cdot e_2 = 0$ for some $p^1, p^2 \in \overline{\omega}_i$.
Then we have
\begin{equation*}
\| v \|_{L^2 (\omega_i)}
\lesssim H | v |_{H^1 (\omega_i)} + H^2 | v |_{H^2 (\omega_i)}
\lesssim  \frac{H^2}{| \sin \alpha |} \left( 1 + \log \frac{H}{h} \right)^\frac{1}{2} | v |_{H^2 (\omega_i)},
\end{equation*}
where $\alpha$ is the angle formed by $e_1$ and $e_2$.
\end{lemma}
\begin{proof}
Let $c = \frac{1}{|\omega_i|} \int_{\omega_i} v \,dx$.
One can easily check that $v (p^0) = 0$ implies that $|c| \leq \| v - c \|_{L^{\infty} (\omega_i)}$~(cf.~\cite[equation~(3.22)]{BPS:1986}).
Since $\tS_h \subset H^2 (\Omega)$, by combining the Sobolev inequality~\cite[Theorem~1.4.6]{BS:2008} and a scaling argument~\cite[Section~3.4]{TW:2005}, we obtain
\begin{equation*}
\begin{split}
    | c | &\leq \| v - c \|_{L^{\infty}(\omega_i)} \\
    &\lesssim H^{-1} \left( \| v - c \|_{L^2 (\omega_i)} +  H | v - c |_{H^1 (\omega_i)} + H^2 | v - c |_{H^2 (\omega_i)} \right) \\
    &= H^{-1} \left( \| v - c \|_{L^2 (\omega_i)} + H | v |_{H^1 (\omega_i)} + H^2 | v |_{H^2 (\omega_i)} \right).
\end{split}
\end{equation*}
Then it follows that
\begin{equation} \begin{split}
\label{Lem:PF_1}
    \| v \|_{L^2 (\omega_i)} &\leq \| v - c \|_{L^2 (\omega_i)} + \| c \|_{L^2 (\omega_i)} \\
    &\lesssim \| v - c \|_{L^2 (\omega_i)} + H | c | \\
    &\lesssim \| v - c \|_{L^2 (\omega_i)} + H | v |_{H^1 (\omega_i)} + H^2 | v |_{H^2 (\omega_i)} \\
    &\lesssim H | v |_{H^1 (\omega_i)} + H^2 | v |_{H^2 (\omega_i)},
\end{split} \end{equation}
where the last inequality is due to the Poincar\'{e}--Friedrichs inequality.
Meanwhile, invoking the discrete Poincar\'{e}--Friedrichs inequality presented in~\cite[Lemma~5.1]{BL:2005} and Lemma~\ref{Lem:H1}, we get
\begin{equation} \begin{split}
\label{Lem:PF_2}
&\| \nabla v \cdot e_1 \|_{L^2 (\omega_i)}^2 + \| \nabla v \cdot e_2 \|_{L^2 (\omega_i)}^2 \\
&\lesssim H^2 \left( 1 + \log \frac{H}{h} \right) \left( \| \nabla ( \nabla v \cdot e_1 ) \|_{L^2 (\omega_i)}^2 + \| \nabla ( \nabla v  \cdot e_2) \|_{L^2 (\omega_i)}^2 \right) \\
&\lesssim H^2 \left( 1 + \log \frac{H}{h} \right) | v |_{H^2 (\omega_i)}^2. 
\end{split} \end{equation}
Combining~\eqref{Lem:PF_1}, Lemma~\ref{Lem:H1}, and~\eqref{Lem:PF_2} yields
\begin{equation*} \begin{split}
\| v \|_{L^2 (\omega_i)}^2 &\lesssim H^2 | v |_{H^1 (\omega_i)}^2 + H^4 | v |_{H^2 (\omega_i)}^2 \\
&\lesssim \frac{H^2}{\sin^2 \alpha} \left( \| \nabla v \cdot e_1 \|_{L^2 (\omega_i)}^2 + \| \nabla v \cdot e_2 \|_{L^2 (\omega_i)}^2 \right) + H^4 | v |_{H^2 (\omega_i)}^2 \\
& \lesssim \frac{H^4}{\sin^2 \alpha} \left( 1 + \log \frac{H}{h} \right) | v |_{H^2 (\omega_i)}^2,
\end{split} \end{equation*}
which is our desired result.
\end{proof}

In what follows, we write the collection of all $\cT_h$-vertices in $\overline{\omega}_i$ as $\cV_h^i$.
In addition, we introduce the notion of \textit{$\alpha$-biasedness} in Definition~\ref{Def:biased}.

\begin{definition}
\label{Def:biased}
For $\alpha \in [0, \pi/2)$ and a region $\omega_i$, $i \in \cI_H$, we say that a function $v \in \tS_h$ is $\alpha$-biased on $\omega_i$ if, for any nonparallel unit vectors $e_1, e_2 \in \mathbb{R}^2$ such that $\nabla v (p^1) \cdot e_1 =\nabla v (p^2) \cdot e_2 = 0$ for some $p^1, p^2 \in \overline{\omega}_i$, the acute angle formed by $e_1$ and $e_2$ does not exceed $\alpha / 2$.
\end{definition}

If we set $\alpha = 0$ in Definition~\ref{Def:biased}, then this implies that the gradient of $u$ is parallel at every point in $\overline{\omega}_i$.
Hence, $\alpha$-biasedness intuitively suggests that $u$ is nearly constant along some direction and varies predominantly along its perpendicular direction.
The following lemma summarizes a useful property of $\alpha$-biased functions, which plays an important role in the proof of Lemma~\ref{Lem:J_i}.

\begin{lemma}
\label{Lem:linear}
Suppose that Assumption~\ref{Ass:cT} holds.
For any $\alpha (H,h)$-biased function $v \in \tS_h$ on a region $\omega_i$, $i \in \cI_H$, there exists a linear function $\ell \in \mathbb{P}_1 (\overline{\omega}_i)$ that satisfies the following:
\begin{subequations}
\begin{align}
\label{Lem:linear_1}
\ell (p^1) = v (p^1) \quad \text{ for some } p^1 \in \cV_h^i, \\
\label{Lem:linear_2}
\begin{cases}
    \displaystyle 0 \leq \ell \leq v, &  \text{ if } v > 0, \\
    \displaystyle v \leq \ell \leq 0, &  \text{ if } v < 0, \\
    \displaystyle \ell = 0, & \text{ if } v = 0,
\end{cases}
 \quad \text{ on } \cV_h^i.
 \end{align}
\end{subequations}
\end{lemma}
\begin{proof}
Throughout this proof, we write $\alpha = \alpha (H,h)$.
Let $p^0 \in \arg\min_{\cV_h^i} | v |$.
Since $v$ is $\alpha$-biased, we deduce that $\nabla v (p^0) \neq 0$, implying that the zero level set of $v - v (p^0)$ has a tangent line $T$ at $p^0$.
Let $T_{\alpha}$ be a double cone in $\overline{\omega}_i$ centered at $p^0$, with $T$ as its axis and aperture $\alpha$, i.e.,
\begin{equation*}
T_{\alpha} = \left\{ x \in \overline{\omega}_i : \text{ the angle between } \overline{xp^0} \text{ and } T \text{ does not exceed } \frac{\alpha}{2} \right\}.
\end{equation*}
We observe that all points in $\cV_h^i \cap T_{\alpha}$ must be collinear; otherwise, an acute angle at $p^0$ formed by noncollinear $\cT_h$-vertices is less than $\alpha$, which contradicts Assumption~\ref{Ass:cT}.
Hence, we have a line $L$ inside $T_{\alpha}$ that passes through all points in $\cV_h^i \cap T_{\alpha}$.
Let $e_L \in \mathbb{R}^2$ denote the unit normal vector of $L$ satisfying $\nabla v (p^0) \cdot e_L > 0$.
Now, we make the following assumption: for $p \in \cV_h^i \setminus L$, we have
\begin{equation}
\label{separate}
\begin{split}
&\text{if } v (p) > 0 , \quad \text{ then } (p - p^0) \cdot e_L > 0, \\
&\text{if } v (p) < 0 , \quad \text{ then } (p - p^0) \cdot e_L < 0.
\end{split}
\end{equation}
Under the assumption~\eqref{separate}, we define a linear function $\ell \in \mathbb{P}_1 (\overline{\omega}_i)$ as follows:
\begin{equation*}
\ell (x) = \left( \min_{p \in \cV_h^i \setminus L} \frac{v (p)}{(p - p^0) \cdot e_L} \right) (x-p^0) \cdot e_L, \quad x \in \overline{\omega}_i.
\end{equation*}
Then~\eqref{Lem:linear_1} holds with
\begin{equation*}
p^1 \in \argmin_{p \in \cV_h^i \setminus L} \frac{v (p)}{(p - p^0) \cdot e_L}.
\end{equation*}
Moreover, we readily obtain~\eqref{Lem:linear_2} from the definition of $\ell$ and~\eqref{separate}.

We still need to provide a proof of~\eqref{separate}.
We introduce the polar coordinate such that, the origin is $p^0$, the $x_1$-axis agrees with $T$, and $x_2$-axis has the same direction as $\nabla v (p^0)$.
In this coordinate, let $e_{\theta}$ denote the unit vector toward the angle $\theta \in \mathbb{R}$.
As the line $L$ is inside $T_{\alpha}$, it is parallel to $e_{\xi}$ for some $\xi \in [-\alpha/2, \alpha/2]$.
Since $v$ is $\alpha$-biased and $\nabla v (p^0) \cdot e_0 = 0$, we deduce that
\begin{equation*}
\nabla v (x) \cdot e_{\theta} \neq 0,
\quad x \in \overline{\omega}_i,
\gap \theta \in \left( \frac{\alpha}{2}, \pi - \frac{\alpha}{2} \right) \cup \left( \pi + \frac{\alpha}{2}, 2\pi - \frac{\alpha}{2} \right).
\end{equation*}
Moreover, since $\nabla v (p^0) \cdot e_{\pi/2} > 0$ and $\nabla v (p^0) \cdot e_{3\pi/2} < 0$, by the continuity of $\nabla v$, we have
\begin{equation}
\label{direction_sign}
\begin{cases}
\nabla v \cdot e_{\theta} > 0, & \text{ if } \theta \in \left( \frac{\alpha}{2}, \pi - \frac{\alpha}{2} \right), \\
\nabla v \cdot e_{\theta} < 0, & \text{ if } \theta \in  \left( \pi + \frac{\alpha}{2}, 2\pi - \frac{\alpha}{2} \right).
\end{cases}
\end{equation}
Now, take any $p \in \cV_h^i \setminus L$ such that $v (p) > 0$.
By the definition of $p^0$, we have $v (p) \geq v (p^0)$.
Moreover, by invoking the mean value theorem, convexity of $\omega_i$, and the $\alpha$-biasedness of $v$, we observe that $v (p) \neq v (p^0)$.
Hence, we deduce that $v (p) > v(p^0)$.
By invoking the mean value theorem and the convexity of $\omega_i$ again, we obtain
\begin{equation*}
    0 < v (p) - v (p^0) = \nabla v (p') \cdot (p - p^0)
\end{equation*}
for some $p' \in \omega_i$.
Then~\eqref{direction_sign} implies that the direction of $p - p^0$ aligns with  $e_{\eta}$ for some $\eta \in (\alpha/2 , \pi - \alpha/2)$.
Noting that $e_L = e_{\xi + \pi/2}$, the fact that the difference between $\eta$ and $\xi + \pi / 2$ is less than $\pi / 2$ implies that $(p - p^0) \cdot e_L > 0$.
Similarly, if we take any $p \in \cV_h^i \setminus L$ such that $v (p) < 0$, then we can prove that $(p - p^0) \cdot e_L < 0$, which completes the proof of~\eqref{separate}.
\end{proof}

Finally, we provide the proof of Lemma~\ref{Lem:J_i} in the following.

\begin{proof}[Proof of Lemma~\ref{Lem:J_i}]
Throughout this proof, we write $\alpha = \alpha (H,h)$.
If $v$ is not $\alpha$-biased on $\omega_i$, then we set
\begin{equation*}
J_i v = v (q), \quad q \in \argmin_{\overline{\omega}_i} | v |.
\end{equation*}
Otherwise, invoking Lemma~\ref{Lem:linear}, we obtain a linear function $\ell_1 \in \mathbb{P}_1 (\overline{\omega}_i)$ such that $( v - \ell_1) (p^1) = 0$ for some $p^1 \in \cV_h^i$ and
\begin{equation}
\label{Lem3:J_i}
\begin{cases}
    \displaystyle 0 \leq \ell_1 \leq v, &  \text{ if } v > 0, \\
    \displaystyle v \leq \ell_1 \leq 0, &  \text{ if } v < 0, \\
    \displaystyle \ell_1 = 0, & \text{ if } v = 0, \\
\end{cases}
 \quad \text{ on } \cV_h^i.
\end{equation}
If $v - \ell_1$ is not $\alpha$-biased on $\omega_i$, then we set $J_i v = \ell_1$.
Otherwise, invoking Lemma~\ref{Lem:linear} again, we can find a linear function $\ell_2 \in \mathbb{P}_1 (\overline{\omega}_i)$ such that $(v - \ell_1 - \ell_2) (p^1) = (v - \ell_1 - \ell_2) (p^2 ) = 0$ for some $p^2 \in \cV_h^i \setminus \{ p^1 \}$ and
\begin{equation}
\label{Lem4:J_i}
\begin{cases}
    \displaystyle 0 \leq \ell_2 \leq v - \ell_1, &  \text{ if } v - \ell_1 > 0, \\
    \displaystyle v - \ell_1 \leq \ell_2 \leq 0, &  \text{ if } v - \ell_1 < 0, \\
    \displaystyle \ell_2 = 0, & \text{ if } v - \ell_1 = 0, \\
\end{cases}
 \quad \text{ on } \cV_h^i.
\end{equation}
We repeat the above argument once again; if $v - \ell_1 - \ell_2$ is not $\alpha$-biased on $\omega_i$, then we set $J_i v = \ell_1 + \ell_2$.
Otherwise, using Lemma~\ref{Lem:linear} yields a linear function $\ell_3 \in \mathbb{P}_1 ( \overline{\omega}_i)$ such that $v - \ell_1 - \ell_2 - \ell_3$ vanishes at three noncollinear vertices $p^1, p^2$, and $p^3 \in \cV_h^i$ and
\begin{equation}
\label{Lem5:J_i}
\begin{cases}
    \displaystyle 0 \leq \ell_3 \leq v - \ell_1 - \ell_2, &  \text{ if } v - \ell_1 - \ell_2 > 0, \\
    \displaystyle v - \ell_1 - \ell_2 \leq \ell_3 \leq 0, &  \text{ if } v - \ell_1 - \ell_2 < 0, \\
    \displaystyle \ell_3 = 0, & \text{ if } v - \ell_1 - \ell_2 = 0, \\
\end{cases}
 \quad \text{ on } \cV_h^i.
\end{equation}
Note that, by Assumption~\ref{Ass:cT}, the sine of the angle formed by $p^1$, $p^2$, and $p^3$ is greater than or equal to $\sin \alpha$.
Hence, by invoking the mean value theorem and the convexity of $\omega_i$, we deduce that $v - \ell_1 - \ell_2 - \ell_3$ is not $\alpha$-biased on $\omega_i$.
Finally, we set $J_i v = \ell_1 + \ell_2 + \ell_3$.
In all cases,~\eqref{Lem1:J_i} is easily verified by using~\eqref{Lem3:J_i},~\eqref{Lem4:J_i}, and~\eqref{Lem5:J_i}.
In addition, invoking Lemma~\ref{Lem:PF} yields~\eqref{Lem2:J_i}.

It remains to verify~\eqref{J_i_BC}.
If $v = 0$ on $\Gamma$, then~\eqref{Lem1:J_i} ensures that $J_i v = 0$ on $\Gamma$.
On the other hand, $v = \frac{\partial v}{\partial \nu} = 0$ on $\Gamma$ implies that $v$ is not $\alpha$-biased on $\omega_i$, so that we have $J_i v$ vanishes on $\omega_i$.
\end{proof} 

\section{Conclusion}
\label{Sec:Conclusion}
In this paper, we conducted a rigorous convergence analysis of additive Schwarz methods for fourth-order variational inequalities.
We introduced a novel coarse space along with a corresponding nonlinear positive interpolation operator, which make the two-level method scalable.
This ensures that the convergence rate remains uniformly bounded when $H/h$ and $H/\delta$ are held constant.
Our theoretical results were verified through numerical experiments conducted on two examples: the displacement obstacle problem of clamped plates and the elliptic distributed optimal control problem.

This paper leaves us an interesting topic for future research.
Although Theorem~\ref{Thm:2L} is sufficient to ensure the scalability of the two-level method, it is unclear whether the estimate in Theorem~\ref{Thm:2L} is optimal.
Investigating the sharpness of Theorem~\ref{Thm:2L} or seeking to refine the estimate may require the development of new mathematical tools, which could be explored in future work.

\section*{Acknowledgement}
This work was inspired by research discussions with Professor Chang-Ock Lee on efficient numerical methods for elliptic distributed optimal control problems.
The author wishes to express gratitude to him for his assistance in preparing this manuscript.

\bibliographystyle{siamplain}
\bibliography{refs_Schwarz_4VI}
\end{document}